\documentclass{amsart}
\pdfoutput=1

\usepackage[T1]{fontenc}

\usepackage{microtype}
\usepackage{tikz}
\usepackage[pdftex,colorlinks,citecolor=black,linkcolor=black,urlcolor=black,bookmarks=false]{hyperref}

\usepackage{booktabs}
\usepackage{makecell}

\def\arXiv#1{arXiv:\href{http://arXiv.org/abs/#1}{#1}}
\usepackage{doi}

\newtheorem{theorem}{Theorem}[section]
\newtheorem{conjecture}[theorem]{Conjecture}

\numberwithin{figure}{section}
\numberwithin{table}{section}
\numberwithin{equation}{section}

\newcommand{\C}{\mathbb{C}}
\newcommand{\R}{\mathbb{R}}
\newcommand{\Z}{\mathbb{Z}}
\newcommand{\F}{\mathbb{F}}
\newcommand{\Q}{\mathbb{Q}}
\newcommand{\Smat}{\mathbb{S}}

\date{March 25, 2024}

\title[Optimality of spherical codes]{Optimality of spherical codes via exact semidefinite programming bounds}

\author{Henry Cohn}
\address{Microsoft Research New England, One Memorial Drive, Cambridge, MA
02142, USA}
\email{cohn@microsoft.com}

\author{David de Laat}
\address{Delft Institute of Applied Mathematics, Delft University of Technology, Delft, The Netherlands}
\email{d.delaat@tudelft.nl}

\author{Nando Leijenhorst}
\address{Delft Institute of Applied Mathematics, Delft University of Technology, Delft, The Netherlands}
\email{n.m.leijenhorst@tudelft.nl}

\begin{document}
 
\begin{abstract}
We show that the spectral embeddings of all known triangle-free strongly regular graphs are optimal spherical codes (the new cases are $56$~points in $20$~dimensions, $50$~points in $21$~dimensions, and $77$~points in $21$~dimensions), as are certain mutually unbiased basis arrangements constructed using Kerdock codes in up to $1024$ dimensions (namely, $2^{4k} + 2^{2k+1}$ points in $2^{2k}$ dimensions for $2 \le k \le 5$). As a consequence of the latter, we obtain optimality of the Kerdock binary codes of block length $64$, $256$, and $1024$, as well as uniqueness for block length~$64$. We also prove universal optimality for $288$ points on a sphere in $16$ dimensions. To prove these results, we use three-point semidefinite programming bounds, for which only a few sharp cases were known previously. To obtain rigorous results, we develop improved techniques for rounding approximate solutions of semidefinite programs to produce exact optimal solutions.
\end{abstract}

\maketitle

\section{Introduction}

The spherical code problem asks how to arrange $N$ points on the unit sphere $S^{n-1}$ in $\R^n$ so that the distance between the closest pair of points is maximized. In other words, we seek a subset $\mathcal{C} \subseteq S^{n-1}$ with $|\mathcal{C}|=N$ such that the maximal inner product
\[
t = \max_{\substack{x,y \in \mathcal{C}\\x \ne y}} \, \langle x,y \rangle 
\]
is minimized, or equivalently such that the minimal angle $\theta = \cos^{-1} t$ is maximized;
we can also view this problem as packing $N$ congruent spherical caps with angular radius $\theta/2$ chosen to be as large as possible. A \emph{spherical code} is any nonempty, finite subset $\mathcal{C}$ of $S^{n-1}$, and it is \emph{optimal} if it minimizes $t$ given $n$ and $|\mathcal{C}|$. This optimization problem arises naturally in geometry, information theory, and physics \cite{B04, BG09, C10, C17, CS99, EZ01, FT64}, and it remains unsolved in general. We will henceforth assume that $n>2$, because the cases $n=1$ and $n=2$ are trivial.

When $N \le 2n$, a systematic solution is possible. For $2 \le N \le n+1$, there is a unique optimal code up to isometry, namely a regular simplex centered at the origin, and the optimal value of $t$ is $-1/(N-1)$. For $N=2n$ there is again a unique optimal code, namely the regular cross polytope (equivalently, $n$ orthogonal pairs of antipodal points), and the optimal value of $t$ is $0$. When $n+1 < N < 2n$, the optimal value of $t$ remains $0$, but the optimal code is no longer unique. See Sections~6.1 and~6.2 of \cite{B04} for proofs, as well as the original papers \cite{A52, DH51, R55, S52}.

However, little is known when $N>2n$. For $N=2n+1$, Example~5.10.4 in \cite{EZ01} provides a compelling conjecture for the optimal code, but even the case $N=2n+2$ has resisted a general conjecture. As $N$ grows, finding the best spherical code becomes an increasingly complex and delicate optimization problem, with little hope of any simple description of the exact optimum. (See, for example, \cite{codetables} for the results of numerical optimization.) Even in particularly symmetric cases, such as the code shown in Figure~\ref{figure:plot}, proving optimality is a daunting task.

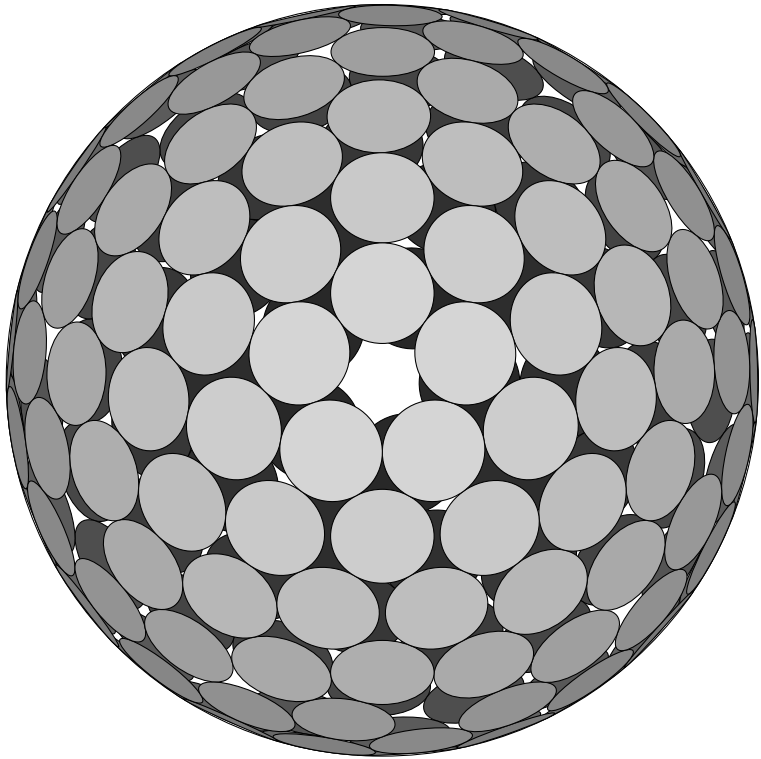
\begin{figure}
\centering
\begin{tikzpicture}
\draw[thin] (0,0) circle (5);
\fill[fill=black!84.200618, rotate around={-19.920335:(0.39502872,1.0900459)}] (0.39502872,1.0900459) ellipse (0.68803348 and 0.66891303);
\draw[thin, rotate around={-19.920335:(0.39502872,1.0900459)}] (0.39502872,1.0900459) ellipse (0.68803348 and 0.66891303);
\fill[fill=black!84.200618, rotate around={88.079665:(1.1587658,-0.038851941)}] (1.1587658,-0.038851941) ellipse (0.68803348 and 0.66891303);
\draw[thin, rotate around={88.079665:(1.1587658,-0.038851941)}] (1.1587658,-0.038851941) ellipse (0.68803348 and 0.66891303);
\fill[fill=black!84.200618, rotate around={16.079665:(0.32112794,-1.1140577)}] (0.32112794,-1.1140577) ellipse (0.68803348 and 0.66891303);
\draw[thin, rotate around={16.079665:(0.32112794,-1.1140577)}] (0.32112794,-1.1140577) ellipse (0.68803348 and 0.66891303);
\fill[fill=black!84.200618, rotate around={-55.920335:(-0.96029783,-0.64967357)}] (-0.96029783,-0.64967357) ellipse (0.68803348 and 0.66891303);
\draw[thin, rotate around={-55.920335:(-0.96029783,-0.64967357)}] (-0.96029783,-0.64967357) ellipse (0.68803348 and 0.66891303);
\fill[fill=black!84.200618, rotate around={52.079665:(-0.91462463,0.71253733)}] (-0.91462463,0.71253733) ellipse (0.68803348 and 0.66891303);
\draw[thin, rotate around={52.079665:(-0.91462463,0.71253733)}] (-0.91462463,0.71253733) ellipse (0.68803348 and 0.66891303);
\fill[fill=black!82.344326, rotate around={-55.920335:(1.7184244,1.1625715)}] (1.7184244,1.1625715) ellipse (0.68803348 and 0.62474596);
\draw[thin, rotate around={-55.920335:(1.7184244,1.1625715)}] (1.7184244,1.1625715) ellipse (0.68803348 and 0.62474596);
\fill[fill=black!82.344326, rotate around={52.079665:(1.6366936,-1.2750644)}] (1.6366936,-1.2750644) ellipse (0.68803348 and 0.62474596);
\draw[thin, rotate around={52.079665:(1.6366936,-1.2750644)}] (1.6366936,-1.2750644) ellipse (0.68803348 and 0.62474596);
\fill[fill=black!82.344326, rotate around={-19.920335:(-0.70689215,-1.9506047)}] (-0.70689215,-1.9506047) ellipse (0.68803348 and 0.62474596);
\draw[thin, rotate around={-19.920335:(-0.70689215,-1.9506047)}] (-0.70689215,-1.9506047) ellipse (0.68803348 and 0.62474596);
\fill[fill=black!82.344326, rotate around={16.079665:(-0.57464889,1.9935731)}] (-0.57464889,1.9935731) ellipse (0.68803348 and 0.62474596);
\draw[thin, rotate around={16.079665:(-0.57464889,1.9935731)}] (-0.57464889,1.9935731) ellipse (0.68803348 and 0.62474596);
\fill[fill=black!82.344326, rotate around={88.079665:(-2.0735770,0.069524394)}] (-2.0735770,0.069524394) ellipse (0.68803348 and 0.62474596);
\draw[thin, rotate around={88.079665:(-2.0735770,0.069524394)}] (-2.0735770,0.069524394) ellipse (0.68803348 and 0.62474596);
\fill[fill=black!81.284273, rotate around={-19.920335:(0.82725244,2.2827279)}] (0.82725244,2.2827279) ellipse (0.68803348 and 0.59967197);
\draw[thin, rotate around={-19.920335:(0.82725244,2.2827279)}] (0.82725244,2.2827279) ellipse (0.68803348 and 0.59967197);
\fill[fill=black!81.284273, rotate around={88.079665:(2.4266384,-0.081362092)}] (2.4266384,-0.081362092) ellipse (0.68803348 and 0.59967197);
\draw[thin, rotate around={88.079665:(2.4266384,-0.081362092)}] (2.4266384,-0.081362092) ellipse (0.68803348 and 0.59967197);
\fill[fill=black!81.284273, rotate around={16.079665:(0.67249254,-2.3330125)}] (0.67249254,-2.3330125) ellipse (0.68803348 and 0.59967197);
\draw[thin, rotate around={16.079665:(0.67249254,-2.3330125)}] (0.67249254,-2.3330125) ellipse (0.68803348 and 0.59967197);
\fill[fill=black!81.284273, rotate around={-55.920335:(-2.0110151,-1.3605189)}] (-2.0110151,-1.3605189) ellipse (0.68803348 and 0.59967197);
\draw[thin, rotate around={-55.920335:(-2.0110151,-1.3605189)}] (-2.0110151,-1.3605189) ellipse (0.68803348 and 0.59967197);
\fill[fill=black!81.284273, rotate around={52.079665:(-1.9153682,1.4921656)}] (-1.9153682,1.4921656) ellipse (0.68803348 and 0.59967197);
\draw[thin, rotate around={52.079665:(-1.9153682,1.4921656)}] (-1.9153682,1.4921656) ellipse (0.68803348 and 0.59967197);
\fill[fill=black!78.514421, rotate around={-69.301933:(2.9157673,1.1016642)}] (2.9157673,1.1016642) ellipse (0.68803348 and 0.53467100);
\draw[thin, rotate around={-69.301933:(2.9157673,1.1016642)}] (2.9157673,1.1016642) ellipse (0.68803348 and 0.53467100);
\fill[fill=black!78.514421, rotate around={-6.5387374:(-0.35494222,-3.0966722)}] (-0.35494222,-3.0966722) ellipse (0.68803348 and 0.53467100);
\draw[thin, rotate around={-6.5387374:(-0.35494222,-3.0966722)}] (-0.35494222,-3.0966722) ellipse (0.68803348 and 0.53467100);
\fill[fill=black!78.514421, rotate around={65.461263:(2.8354271,-1.2944945)}] (2.8354271,-1.2944945) ellipse (0.68803348 and 0.53467100);
\draw[thin, rotate around={65.461263:(2.8354271,-1.2944945)}] (2.8354271,-1.2944945) ellipse (0.68803348 and 0.53467100);
\fill[fill=black!78.514421, rotate around={-42.538737:(2.1073325,2.2966306)}] (2.1073325,2.2966306) ellipse (0.68803348 and 0.53467100);
\draw[thin, rotate around={-42.538737:(2.1073325,2.2966306)}] (2.1073325,2.2966306) ellipse (0.68803348 and 0.53467100);
\fill[fill=black!78.514421, rotate around={2.6980672:(-0.14672331,3.1134924)}] (-0.14672331,3.1134924) ellipse (0.68803348 and 0.53467100);
\draw[thin, rotate around={2.6980672:(-0.14672331,3.1134924)}] (-0.14672331,3.1134924) ellipse (0.68803348 and 0.53467100);
\fill[fill=black!78.514421, rotate around={38.698067:(1.9487666,-2.4326265)}] (1.9487666,-2.4326265) ellipse (0.68803348 and 0.53467100);
\draw[thin, rotate around={38.698067:(1.9487666,-2.4326265)}] (1.9487666,-2.4326265) ellipse (0.68803348 and 0.53467100);
\fill[fill=black!78.514421, rotate around={-33.301933:(-1.7113633,-2.6051101)}] (-1.7113633,-2.6051101) ellipse (0.68803348 and 0.53467100);
\draw[thin, rotate around={-33.301933:(-1.7113633,-2.6051101)}] (-1.7113633,-2.6051101) ellipse (0.68803348 and 0.53467100);
\fill[fill=black!78.514421, rotate around={74.698067:(-3.0064473,0.82257991)}] (-3.0064473,0.82257991) ellipse (0.68803348 and 0.53467100);
\draw[thin, rotate around={74.698067:(-3.0064473,0.82257991)}] (-3.0064473,0.82257991) ellipse (0.68803348 and 0.53467100);
\fill[fill=black!78.514421, rotate around={29.461263:(-1.5330240,2.7138903)}] (-1.5330240,2.7138903) ellipse (0.68803348 and 0.53467100);
\draw[thin, rotate around={29.461263:(-1.5330240,2.7138903)}] (-1.5330240,2.7138903) ellipse (0.68803348 and 0.53467100);
\fill[fill=black!78.514421, rotate around={-78.538737:(-3.0547935,-0.61935423)}] (-3.0547935,-0.61935423) ellipse (0.68803348 and 0.53467100);
\draw[thin, rotate around={-78.538737:(-3.0547935,-0.61935423)}] (-3.0547935,-0.61935423) ellipse (0.68803348 and 0.53467100);
\fill[fill=black!76.378523, rotate around={-20.681763:(1.2404606,3.2859466)}] (1.2404606,3.2859466) ellipse (0.68803348 and 0.48506528);
\draw[thin, rotate around={-20.681763:(1.2404606,3.2859466)}] (1.2404606,3.2859466) ellipse (0.68803348 and 0.48506528);
\fill[fill=black!76.378523, rotate around={87.318237:(3.5084443,-0.16433484)}] (3.5084443,-0.16433484) ellipse (0.68803348 and 0.48506528);
\draw[thin, rotate around={87.318237:(3.5084443,-0.16433484)}] (3.5084443,-0.16433484) ellipse (0.68803348 and 0.48506528);
\fill[fill=black!76.378523, rotate around={15.318237:(0.92787720,-3.3875111)}] (0.92787720,-3.3875111) ellipse (0.68803348 and 0.48506528);
\draw[thin, rotate around={15.318237:(0.92787720,-3.3875111)}] (0.92787720,-3.3875111) ellipse (0.68803348 and 0.48506528);
\fill[fill=black!76.378523, rotate around={-56.681763:(-2.9349847,-1.9292622)}] (-2.9349847,-1.9292622) ellipse (0.68803348 and 0.48506528);
\draw[thin, rotate around={-56.681763:(-2.9349847,-1.9292622)}] (-2.9349847,-1.9292622) ellipse (0.68803348 and 0.48506528);
\fill[fill=black!76.378523, rotate around={51.318237:(-2.7417975,2.1951615)}] (-2.7417975,2.1951615) ellipse (0.68803348 and 0.48506528);
\draw[thin, rotate around={51.318237:(-2.7417975,2.1951615)}] (-2.7417975,2.1951615) ellipse (0.68803348 and 0.48506528);
\fill[fill=black!73.826559, rotate around={-19.920335:(-1.3242544,-3.6541598)}] (-1.3242544,-3.6541598) ellipse (0.68803348 and 0.42640081);
\draw[thin, rotate around={-19.920335:(-1.3242544,-3.6541598)}] (-1.3242544,-3.6541598) ellipse (0.68803348 and 0.42640081);
\fill[fill=black!73.826559, rotate around={-55.920335:(3.2192056,2.1779002)}] (3.2192056,2.1779002) ellipse (0.68803348 and 0.42640081);
\draw[thin, rotate around={-55.920335:(3.2192056,2.1779002)}] (3.2192056,2.1779002) ellipse (0.68803348 and 0.42640081);
\fill[fill=black!73.826559, rotate around={16.079665:(-1.0765169,3.7346546)}] (-1.0765169,3.7346546) ellipse (0.68803348 and 0.42640081);
\draw[thin, rotate around={16.079665:(-1.0765169,3.7346546)}] (-1.0765169,3.7346546) ellipse (0.68803348 and 0.42640081);
\fill[fill=black!73.826559, rotate around={52.079665:(3.0660954,-2.3886383)}] (3.0660954,-2.3886383) ellipse (0.68803348 and 0.42640081);
\draw[thin, rotate around={52.079665:(3.0660954,-2.3886383)}] (3.0660954,-2.3886383) ellipse (0.68803348 and 0.42640081);
\fill[fill=black!73.826559, rotate around={88.079665:(-3.8845296,0.13024333)}] (-3.8845296,0.13024333) ellipse (0.68803348 and 0.42640081);
\draw[thin, rotate around={88.079665:(-3.8845296,0.13024333)}] (-3.8845296,0.13024333) ellipse (0.68803348 and 0.42640081);
\fill[fill=black!72.787786, rotate around={-75.539104:(3.8882630,1.0027425)}] (3.8882630,1.0027425) ellipse (0.68803348 and 0.40271314);
\draw[thin, rotate around={-75.539104:(3.8882630,1.0027425)}] (3.8882630,1.0027425) ellipse (0.68803348 and 0.40271314);
\fill[fill=black!72.787786, rotate around={32.460896:(2.1552041,-3.3880934)}] (2.1552041,-3.3880934) ellipse (0.68803348 and 0.40271314);
\draw[thin, rotate around={32.460896:(2.1552041,-3.3880934)}] (2.1552041,-3.3880934) ellipse (0.68803348 and 0.40271314);
\fill[fill=black!72.787786, rotate around={-3.5391039:(0.24787457,4.0078224)}] (0.24787457,4.0078224) ellipse (0.68803348 and 0.40271314);
\draw[thin, rotate around={-3.5391039:(0.24787457,4.0078224)}] (0.24787457,4.0078224) ellipse (0.68803348 and 0.40271314);
\fill[fill=black!72.787786, rotate around={-39.539104:(-2.5562736,-3.0966994)}] (-2.5562736,-3.0966994) ellipse (0.68803348 and 0.40271314);
\draw[thin, rotate around={-39.539104:(-2.5562736,-3.0966994)}] (-2.5562736,-3.0966994) ellipse (0.68803348 and 0.40271314);
\fill[fill=black!72.787786, rotate around={68.460896:(-3.7350681,1.4742279)}] (-3.7350681,1.4742279) ellipse (0.68803348 and 0.40271314);
\draw[thin, rotate around={68.460896:(-3.7350681,1.4742279)}] (-3.7350681,1.4742279) ellipse (0.68803348 and 0.40271314);
\fill[fill=black!72.341577, rotate around={-36.680383:(2.4295300,3.2617946)}] (2.4295300,3.2617946) ellipse (0.68803348 and 0.39257246);
\draw[thin, rotate around={-36.680383:(2.4295300,3.2617946)}] (2.4295300,3.2617946) ellipse (0.68803348 and 0.39257246);
\fill[fill=black!72.341577, rotate around={-72.680383:(-3.8827658,-1.2108054)}] (-3.8827658,-1.2108054) ellipse (0.68803348 and 0.39257246);
\draw[thin, rotate around={-72.680383:(-3.8827658,-1.2108054)}] (-3.8827658,-1.2108054) ellipse (0.68803348 and 0.39257246);
\fill[fill=black!72.341577, rotate around={71.319617:(3.8529171,-1.3026704)}] (3.8529171,-1.3026704) ellipse (0.68803348 and 0.39257246);
\draw[thin, rotate around={71.319617:(3.8529171,-1.3026704)}] (3.8529171,-1.3026704) ellipse (0.68803348 and 0.39257246);
\fill[fill=black!72.341577, rotate around={-0.68038293:(-0.048296264,-4.0668892)}] (-0.048296264,-4.0668892) ellipse (0.68803348 and 0.39257246);
\draw[thin, rotate around={-0.68038293:(-0.048296264,-4.0668892)}] (-0.048296264,-4.0668892) ellipse (0.68803348 and 0.39257246);
\fill[fill=black!72.341577, rotate around={35.319617:(-2.3513850,3.3185703)}] (-2.3513850,3.3185703) ellipse (0.68803348 and 0.39257246);
\draw[thin, rotate around={35.319617:(-2.3513850,3.3185703)}] (-2.3513850,3.3185703) ellipse (0.68803348 and 0.39257246);
\fill[fill=black!69.299042, rotate around={-55.846592:(-3.6155138,-2.4528050)}] (-3.6155138,-2.4528050) ellipse (0.68803348 and 0.32398662);
\draw[thin, rotate around={-55.846592:(-3.6155138,-2.4528050)}] (-3.6155138,-2.4528050) ellipse (0.68803348 and 0.32398662);
\fill[fill=black!69.299042, rotate around={52.153408:(-3.4500114,2.6805995)}] (-3.4500114,2.6805995) ellipse (0.68803348 and 0.32398662);
\draw[thin, rotate around={52.153408:(-3.4500114,2.6805995)}] (-3.4500114,2.6805995) ellipse (0.68803348 and 0.32398662);
\fill[fill=black!69.299042, rotate around={88.153408:(4.3667347,-0.14078473)}] (4.3667347,-0.14078473) ellipse (0.68803348 and 0.32398662);
\draw[thin, rotate around={88.153408:(4.3667347,-0.14078473)}] (4.3667347,-0.14078473) ellipse (0.68803348 and 0.32398662);
\fill[fill=black!69.299042, rotate around={-19.846592:(1.4832895,4.1095067)}] (1.4832895,4.1095067) ellipse (0.68803348 and 0.32398662);
\draw[thin, rotate around={-19.846592:(1.4832895,4.1095067)}] (1.4832895,4.1095067) ellipse (0.68803348 and 0.32398662);
\fill[fill=black!69.299042, rotate around={16.153408:(1.2155010,-4.1965164)}] (1.2155010,-4.1965164) ellipse (0.68803348 and 0.32398662);
\draw[thin, rotate around={16.153408:(1.2155010,-4.1965164)}] (1.2155010,-4.1965164) ellipse (0.68803348 and 0.32398662);
\fill[fill=black!67.063600, rotate around={-49.398935:(3.4486087,2.9559263)}] (3.4486087,2.9559263) ellipse (0.68803348 and 0.27422804);
\draw[thin, rotate around={-49.398935:(3.4486087,2.9559263)}] (3.4486087,2.9559263) ellipse (0.68803348 and 0.27422804);
\fill[fill=black!67.063600, rotate around={-13.398935:(-1.0525331,-4.4184360)}] (-1.0525331,-4.4184360) ellipse (0.68803348 and 0.27422804);
\draw[thin, rotate around={-13.398935:(-1.0525331,-4.4184360)}] (-1.0525331,-4.4184360) ellipse (0.68803348 and 0.27422804);
\fill[fill=black!67.063600, rotate around={22.601065:(-1.7455743,4.1932532)}] (-1.7455743,4.1932532) ellipse (0.68803348 and 0.27422804);
\draw[thin, rotate around={22.601065:(-1.7455743,4.1932532)}] (-1.7455743,4.1932532) ellipse (0.68803348 and 0.27422804);
\fill[fill=black!67.063600, rotate around={-85.398935:(-4.5274329,-0.36435334)}] (-4.5274329,-0.36435334) ellipse (0.68803348 and 0.27422804);
\draw[thin, rotate around={-85.398935:(-4.5274329,-0.36435334)}] (-4.5274329,-0.36435334) ellipse (0.68803348 and 0.27422804);
\fill[fill=black!67.063600, rotate around={58.601065:(3.8769317,-2.3663903)}] (3.8769317,-2.3663903) ellipse (0.68803348 and 0.27422804);
\draw[thin, rotate around={58.601065:(3.8769317,-2.3663903)}] (3.8769317,-2.3663903) ellipse (0.68803348 and 0.27422804);
\fill[fill=black!65.324654, rotate around={-76.899450:(4.5311663,1.0544813)}] (4.5311663,1.0544813) ellipse (0.68803348 and 0.23590059);
\draw[thin, rotate around={-76.899450:(4.5311663,1.0544813)}] (4.5311663,1.0544813) ellipse (0.68803348 and 0.23590059);
\fill[fill=black!65.324654, rotate around={-4.8994498:(0.39733608,4.6352479)}] (0.39733608,4.6352479) ellipse (0.68803348 and 0.23590059);
\draw[thin, rotate around={-4.8994498:(0.39733608,4.6352479)}] (0.39733608,4.6352479) ellipse (0.68803348 and 0.23590059);
\fill[fill=black!65.324654, rotate around={-40.899450:(-3.0459820,-3.5164460)}] (-3.0459820,-3.5164460) ellipse (0.68803348 and 0.23590059);
\draw[thin, rotate around={-40.899450:(-3.0459820,-3.5164460)}] (-3.0459820,-3.5164460) ellipse (0.68803348 and 0.23590059);
\fill[fill=black!65.324654, rotate around={31.100550:(2.4030787,-3.9835426)}] (2.4030787,-3.9835426) ellipse (0.68803348 and 0.23590059);
\draw[thin, rotate around={31.100550:(2.4030787,-3.9835426)}] (2.4030787,-3.9835426) ellipse (0.68803348 and 0.23590059);
\fill[fill=black!65.324654, rotate around={67.100550:(-4.2855991,1.8102594)}] (-4.2855991,1.8102594) ellipse (0.68803348 and 0.23590059);
\draw[thin, rotate around={67.100550:(-4.2855991,1.8102594)}] (-4.2855991,1.8102594) ellipse (0.68803348 and 0.23590059);
\fill[fill=black!64.415526, rotate around={-69.390419:(-4.4011343,-1.6551177)}] (-4.4011343,-1.6551177) ellipse (0.68803348 and 0.21599723);
\draw[thin, rotate around={-69.390419:(-4.4011343,-1.6551177)}] (-4.4011343,-1.6551177) ellipse (0.68803348 and 0.21599723);
\fill[fill=black!64.415526, rotate around={38.609581:(-2.9341357,3.6742680)}] (-2.9341357,3.6742680) ellipse (0.68803348 and 0.21599723);
\draw[thin, rotate around={38.609581:(-2.9341357,3.6742680)}] (-2.9341357,3.6742680) ellipse (0.68803348 and 0.21599723);
\fill[fill=black!64.415526, rotate around={2.6095814:(0.21408513,-4.6971870)}] (0.21408513,-4.6971870) ellipse (0.68803348 and 0.21599723);
\draw[thin, rotate around={2.6095814:(0.21408513,-4.6971870)}] (0.21408513,-4.6971870) ellipse (0.68803348 and 0.21599723);
\fill[fill=black!64.415526, rotate around={74.609581:(4.5334462,-1.2479035)}] (4.5334462,-1.2479035) ellipse (0.68803348 and 0.21599723);
\draw[thin, rotate around={74.609581:(4.5334462,-1.2479035)}] (4.5334462,-1.2479035) ellipse (0.68803348 and 0.21599723);
\fill[fill=black!64.415526, rotate around={-33.390419:(2.5877387,3.9259402)}] (2.5877387,3.9259402) ellipse (0.68803348 and 0.21599723);
\draw[thin, rotate around={-33.390419:(2.5877387,3.9259402)}] (2.5877387,3.9259402) ellipse (0.68803348 and 0.21599723);
\fill[fill=black!61.671473, rotate around={81.966728:(-4.7753075,0.67395372)}] (-4.7753075,0.67395372) ellipse (0.68803348 and 0.15649243);
\draw[thin, rotate around={81.966728:(-4.7753075,0.67395372)}] (-4.7753075,0.67395372) ellipse (0.68803348 and 0.15649243);
\fill[fill=black!61.671473, rotate around={-62.033272:(4.2594450,2.2616153)}] (4.2594450,2.2616153) ellipse (0.68803348 and 0.15649243);
\draw[thin, rotate around={-62.033272:(4.2594450,2.2616153)}] (4.2594450,2.2616153) ellipse (0.68803348 and 0.15649243);
\fill[fill=black!61.671473, rotate around={-26.033272:(-2.1166192,-4.3333242)}] (-2.1166192,-4.3333242) ellipse (0.68803348 and 0.15649243);
\draw[thin, rotate around={-26.033272:(-2.1166192,-4.3333242)}] (-2.1166192,-4.3333242) ellipse (0.68803348 and 0.15649243);
\fill[fill=black!61.671473, rotate around={9.9667284:(-0.83468309,4.7498505)}] (-0.83468309,4.7498505) ellipse (0.68803348 and 0.15649243);
\draw[thin, rotate around={9.9667284:(-0.83468309,4.7498505)}] (-0.83468309,4.7498505) ellipse (0.68803348 and 0.15649243);
\fill[fill=black!61.671473, rotate around={45.966728:(3.4671649,-3.3520953)}] (3.4671649,-3.3520953) ellipse (0.68803348 and 0.15649243);
\draw[thin, rotate around={45.966728:(3.4671649,-3.3520953)}] (3.4671649,-3.3520953) ellipse (0.68803348 and 0.15649243);
\fill[fill=black!60.704613, rotate around={17.339291:(1.4469656,-4.6344756)}] (1.4469656,-4.6344756) ellipse (0.68803348 and 0.13573349);
\draw[thin, rotate around={17.339291:(1.4469656,-4.6344756)}] (1.4469656,-4.6344756) ellipse (0.68803348 and 0.13573349);
\fill[fill=black!60.704613, rotate around={89.339291:(4.8547852,-0.055985625)}] (4.8547852,-0.055985625) ellipse (0.68803348 and 0.13573349);
\draw[thin, rotate around={89.339291:(4.8547852,-0.055985625)}] (4.8547852,-0.055985625) ellipse (0.68803348 and 0.13573349);
\fill[fill=black!60.704613, rotate around={53.339291:(-3.8946962,2.8988645)}] (-3.8946962,2.8988645) ellipse (0.68803348 and 0.13573349);
\draw[thin, rotate around={53.339291:(-3.8946962,2.8988645)}] (-3.8946962,2.8988645) ellipse (0.68803348 and 0.13573349);
\fill[fill=black!60.704613, rotate around={-54.660709:(-3.9605113,-2.8082778)}] (-3.9605113,-2.8082778) ellipse (0.68803348 and 0.13573349);
\draw[thin, rotate around={-54.660709:(-3.9605113,-2.8082778)}] (-3.9605113,-2.8082778) ellipse (0.68803348 and 0.13573349);
\fill[fill=black!60.704613, rotate around={-18.660709:(1.5534566,4.5998746)}] (1.5534566,4.5998746) ellipse (0.68803348 and 0.13573349);
\draw[thin, rotate around={-18.660709:(1.5534566,4.5998746)}] (1.5534566,4.5998746) ellipse (0.68803348 and 0.13573349);
\fill[fill=black!58.748223, rotate around={-46.863436:(3.5799863,3.3543797)}] (3.5799863,3.3543797) ellipse (0.68803348 and 0.094066364);
\draw[thin, rotate around={-46.863436:(3.5799863,3.3543797)}] (3.5799863,3.3543797) ellipse (0.68803348 and 0.094066364);
\fill[fill=black!58.748223, rotate around={-10.863436:(-0.92461485,-4.8180133)}] (-0.92461485,-4.8180133) ellipse (0.68803348 and 0.094066364);
\draw[thin, rotate around={-10.863436:(-0.92461485,-4.8180133)}] (-0.92461485,-4.8180133) ellipse (0.68803348 and 0.094066364);
\fill[fill=black!58.748223, rotate around={25.136564:(-2.0839281,4.4413296)}] (-2.0839281,4.4413296) ellipse (0.68803348 and 0.094066364);
\draw[thin, rotate around={25.136564:(-2.0839281,4.4413296)}] (-2.0839281,4.4413296) ellipse (0.68803348 and 0.094066364);
\fill[fill=black!58.748223, rotate around={-82.863436:(-4.8679247,-0.60948702)}] (-4.8679247,-0.60948702) ellipse (0.68803348 and 0.094066364);
\draw[thin, rotate around={-82.863436:(-4.8679247,-0.60948702)}] (-4.8679247,-0.60948702) ellipse (0.68803348 and 0.094066364);
\fill[fill=black!58.748223, rotate around={61.136564:(4.2964813,-2.3682090)}] (4.2964813,-2.3682090) ellipse (0.68803348 and 0.094066364);
\draw[thin, rotate around={61.136564:(4.2964813,-2.3682090)}] (4.2964813,-2.3682090) ellipse (0.68803348 and 0.094066364);
\fill[fill=black!56.735421, rotate around={-75.320387:(4.7772447,1.2514700)}] (4.7772447,1.2514700) ellipse (0.68803348 and 0.051677984);
\draw[thin, rotate around={-75.320387:(4.7772447,1.2514700)}] (4.7772447,1.2514700) ellipse (0.68803348 and 0.051677984);
\fill[fill=black!56.735421, rotate around={-3.3203873:(0.28603110,4.9301552)}] (0.28603110,4.9301552) ellipse (0.68803348 and 0.051677984);
\draw[thin, rotate around={-3.3203873:(0.28603110,4.9301552)}] (0.28603110,4.9301552) ellipse (0.68803348 and 0.051677984);
\fill[fill=black!56.735421, rotate around={-39.320387:(-3.1292765,-3.8204545)}] (-3.1292765,-3.8204545) ellipse (0.68803348 and 0.051677984);
\draw[thin, rotate around={-39.320387:(-3.1292765,-3.8204545)}] (-3.1292765,-3.8204545) ellipse (0.68803348 and 0.051677984);
\fill[fill=black!56.735421, rotate around={68.679613:(-4.6004677,1.7955335)}] (-4.6004677,1.7955335) ellipse (0.68803348 and 0.051677984);
\draw[thin, rotate around={68.679613:(-4.6004677,1.7955335)}] (-4.6004677,1.7955335) ellipse (0.68803348 and 0.051677984);
\fill[fill=black!56.735421, rotate around={32.679613:(2.6664685,-4.1567042)}] (2.6664685,-4.1567042) ellipse (0.68803348 and 0.051677984);
\draw[thin, rotate around={32.679613:(2.6664685,-4.1567042)}] (2.6664685,-4.1567042) ellipse (0.68803348 and 0.051677984);
\fill[fill=black!55.723722, rotate around={-67.895863:(-4.5839098,-1.8617183)}] (-4.5839098,-1.8617183) ellipse (0.68803348 and 0.030559709);
\draw[thin, rotate around={-67.895863:(-4.5839098,-1.8617183)}] (-4.5839098,-1.8617183) ellipse (0.68803348 and 0.030559709);
\fill[fill=black!55.723722, rotate around={40.104137:(-3.1871054,3.7842547)}] (-3.1871054,3.7842547) ellipse (0.68803348 and 0.030559709);
\draw[thin, rotate around={40.104137:(-3.1871054,3.7842547)}] (-3.1871054,3.7842547) ellipse (0.68803348 and 0.030559709);
\fill[fill=black!55.723722, rotate around={4.1041368:(0.35409329,-4.9348599)}] (0.35409329,-4.9348599) ellipse (0.68803348 and 0.030559709);
\draw[thin, rotate around={4.1041368:(0.35409329,-4.9348599)}] (0.35409329,-4.9348599) ellipse (0.68803348 and 0.030559709);
\fill[fill=black!55.723722, rotate around={76.104137:(4.8027515,-1.1881928)}] (4.8027515,-1.1881928) ellipse (0.68803348 and 0.030559709);
\draw[thin, rotate around={76.104137:(4.8027515,-1.1881928)}] (4.8027515,-1.1881928) ellipse (0.68803348 and 0.030559709);
\fill[fill=black!55.723722, rotate around={-31.895863:(2.6141704,4.2005163)}] (2.6141704,4.2005163) ellipse (0.68803348 and 0.030559709);
\draw[thin, rotate around={-31.895863:(2.6141704,4.2005163)}] (2.6141704,4.2005163) ellipse (0.68803348 and 0.030559709);
\fill[fill=black!52.760309, rotate around={-60.024472:(4.2857578,2.4719434)}] (4.2857578,2.4719434) ellipse (0.68803348 and 0.030559709);
\draw[thin, rotate around={-60.024472:(4.2857578,2.4719434)}] (4.2857578,2.4719434) ellipse (0.68803348 and 0.030559709);
\fill[fill=black!52.760309, rotate around={83.975528:(-4.9202228,0.51926107)}] (-4.9202228,0.51926107) ellipse (0.68803348 and 0.030559709);
\draw[thin, rotate around={83.975528:(-4.9202228,0.51926107)}] (-4.9202228,0.51926107) ellipse (0.68803348 and 0.030559709);
\fill[fill=black!52.760309, rotate around={11.975528:(-1.0265858,4.8398704)}] (-1.0265858,4.8398704) ellipse (0.68803348 and 0.030559709);
\draw[thin, rotate around={11.975528:(-1.0265858,4.8398704)}] (-1.0265858,4.8398704) ellipse (0.68803348 and 0.030559709);
\fill[fill=black!52.760309, rotate around={-24.024472:(-2.0142791,-4.5189494)}] (-2.0142791,-4.5189494) ellipse (0.68803348 and 0.030559709);
\draw[thin, rotate around={-24.024472:(-2.0142791,-4.5189494)}] (-2.0142791,-4.5189494) ellipse (0.68803348 and 0.030559709);
\fill[fill=black!52.760309, rotate around={47.975528:(3.6753298,-3.3121254)}] (3.6753298,-3.3121254) ellipse (0.68803348 and 0.030559709);
\draw[thin, rotate around={47.975528:(3.6753298,-3.3121254)}] (3.6753298,-3.3121254) ellipse (0.68803348 and 0.030559709);
\fill[fill=black!51.723497, rotate around={19.400052:(1.6403639,-4.6580522)}] (1.6403639,-4.6580522) ellipse (0.68803348 and 0.051677984);
\draw[thin, rotate around={19.400052:(1.6403639,-4.6580522)}] (1.6403639,-4.6580522) ellipse (0.68803348 and 0.051677984);
\fill[fill=black!51.723497, rotate around={55.400052:(-4.0650166,2.8042617)}] (-4.0650166,2.8042617) ellipse (0.68803348 and 0.051677984);
\draw[thin, rotate around={55.400052:(-4.0650166,2.8042617)}] (-4.0650166,2.8042617) ellipse (0.68803348 and 0.051677984);
\fill[fill=black!51.723497, rotate around={-88.599948:(4.9369712,0.12066148)}] (4.9369712,0.12066148) ellipse (0.68803348 and 0.051677984);
\draw[thin, rotate around={-88.599948:(4.9369712,0.12066148)}] (4.9369712,0.12066148) ellipse (0.68803348 and 0.051677984);
\fill[fill=black!51.723497, rotate around={-52.599948:(-3.9231706,-2.9994961)}] (-3.9231706,-2.9994961) ellipse (0.68803348 and 0.051677984);
\draw[thin, rotate around={-52.599948:(-3.9231706,-2.9994961)}] (-3.9231706,-2.9994961) ellipse (0.68803348 and 0.051677984);
\fill[fill=black!51.723497, rotate around={-16.599948:(1.4108521,4.7326251)}] (1.4108521,4.7326251) ellipse (0.68803348 and 0.051677984);
\draw[thin, rotate around={-16.599948:(1.4108521,4.7326251)}] (1.4108521,4.7326251) ellipse (0.68803348 and 0.051677984);
\fill[fill=black!49.621124, rotate around={-9.0568989:(-0.77226838,-4.8447669)}] (-0.77226838,-4.8447669) ellipse (0.68803348 and 0.094066364);
\draw[thin, rotate around={-9.0568989:(-0.77226838,-4.8447669)}] (-0.77226838,-4.8447669) ellipse (0.68803348 and 0.094066364);
\fill[fill=black!49.621124, rotate around={-45.056899:(3.4724608,3.4655708)}] (3.4724608,3.4655708) ellipse (0.68803348 and 0.094066364);
\draw[thin, rotate around={-45.056899:(3.4724608,3.4655708)}] (3.4724608,3.4655708) ellipse (0.68803348 and 0.094066364);
\fill[fill=black!49.621124, rotate around={26.943101:(-2.2229043,4.3734267)}] (-2.2229043,4.3734267) ellipse (0.68803348 and 0.094066364);
\draw[thin, rotate around={26.943101:(-2.2229043,4.3734267)}] (-2.2229043,4.3734267) ellipse (0.68803348 and 0.094066364);
\fill[fill=black!49.621124, rotate around={-81.056899:(-4.8462912,-0.76264445)}] (-4.8462912,-0.76264445) ellipse (0.68803348 and 0.094066364);
\draw[thin, rotate around={-81.056899:(-4.8462912,-0.76264445)}] (-4.8462912,-0.76264445) ellipse (0.68803348 and 0.094066364);
\fill[fill=black!49.621124, rotate around={62.943101:(4.3690031,-2.2315862)}] (4.3690031,-2.2315862) ellipse (0.68803348 and 0.094066364);
\draw[thin, rotate around={62.943101:(4.3690031,-2.2315862)}] (4.3690031,-2.2315862) ellipse (0.68803348 and 0.094066364);
\fill[fill=black!47.525169, rotate around={-73.259627:(4.6493474,1.3984429)}] (4.6493474,1.3984429) ellipse (0.68803348 and 0.13573349);
\draw[thin, rotate around={-73.259627:(4.6493474,1.3984429)}] (4.6493474,1.3984429) ellipse (0.68803348 and 0.13573349);
\fill[fill=black!47.525169, rotate around={-1.2596266:(0.10672916,4.8539348)}] (0.10672916,4.8539348) ellipse (0.68803348 and 0.13573349);
\draw[thin, rotate around={-1.2596266:(0.10672916,4.8539348)}] (0.10672916,4.8539348) ellipse (0.68803348 and 0.13573349);
\fill[fill=black!47.525169, rotate around={-37.259627:(-2.9394170,-3.8641819)}] (-2.9394170,-3.8641819) ellipse (0.68803348 and 0.13573349);
\draw[thin, rotate around={-37.259627:(-2.9394170,-3.8641819)}] (-2.9394170,-3.8641819) ellipse (0.68803348 and 0.13573349);
\fill[fill=black!47.525169, rotate around={70.740373:(-4.5833851,1.6014538)}] (-4.5833851,1.6014538) ellipse (0.68803348 and 0.13573349);
\draw[thin, rotate around={70.740373:(-4.5833851,1.6014538)}] (-4.5833851,1.6014538) ellipse (0.68803348 and 0.13573349);
\fill[fill=black!47.525169, rotate around={34.740373:(2.7667256,-3.9896495)}] (2.7667256,-3.9896495) ellipse (0.68803348 and 0.13573349);
\draw[thin, rotate around={34.740373:(2.7667256,-3.9896495)}] (2.7667256,-3.9896495) ellipse (0.68803348 and 0.13573349);
\fill[fill=black!46.469415, rotate around={42.112936:(-3.2340284,3.5775460)}] (-3.2340284,3.5775460) ellipse (0.68803348 and 0.15649243);
\draw[thin, rotate around={42.112936:(-3.2340284,3.5775460)}] (-3.2340284,3.5775460) ellipse (0.68803348 and 0.15649243);
\fill[fill=black!46.469415, rotate around={-65.887064:(-4.4018182,-1.9702213)}] (-4.4018182,-1.9702213) ellipse (0.68803348 and 0.15649243);
\draw[thin, rotate around={-65.887064:(-4.4018182,-1.9702213)}] (-4.4018182,-1.9702213) ellipse (0.68803348 and 0.15649243);
\fill[fill=black!46.469415, rotate around={6.1129365:(0.51355516,-4.7952097)}] (0.51355516,-4.7952097) ellipse (0.68803348 and 0.15649243);
\draw[thin, rotate around={6.1129365:(0.51355516,-4.7952097)}] (0.51355516,-4.7952097) ellipse (0.68803348 and 0.15649243);
\fill[fill=black!46.469415, rotate around={78.112936:(4.7192127,-0.99338132)}] (4.7192127,-0.99338132) ellipse (0.68803348 and 0.15649243);
\draw[thin, rotate around={78.112936:(4.7192127,-0.99338132)}] (4.7192127,-0.99338132) ellipse (0.68803348 and 0.15649243);
\fill[fill=black!46.469415, rotate around={-29.887064:(2.4030787,4.1812663)}] (2.4030787,4.1812663) ellipse (0.68803348 and 0.15649243);
\draw[thin, rotate around={-29.887064:(2.4030787,4.1812663)}] (2.4030787,4.1812663) ellipse (0.68803348 and 0.15649243);
\fill[fill=black!43.397372, rotate around={-58.529916:(4.0104502,2.4547276)}] (4.0104502,2.4547276) ellipse (0.68803348 and 0.21599723);
\draw[thin, rotate around={-58.529916:(4.0104502,2.4547276)}] (4.0104502,2.4547276) ellipse (0.68803348 and 0.21599723);
\fill[fill=black!43.397372, rotate around={85.470084:(-4.6873750,0.37136715)}] (-4.6873750,0.37136715) ellipse (0.68803348 and 0.21599723);
\draw[thin, rotate around={85.470084:(-4.6873750,0.37136715)}] (-4.6873750,0.37136715) ellipse (0.68803348 and 0.21599723);
\fill[fill=black!43.397372, rotate around={13.470084:(-1.0952874,4.5727173)}] (-1.0952874,4.5727173) ellipse (0.68803348 and 0.21599723);
\draw[thin, rotate around={13.470084:(-1.0952874,4.5727173)}] (-1.0952874,4.5727173) ellipse (0.68803348 and 0.21599723);
\fill[fill=black!43.397372, rotate around={-22.529916:(-1.8016697,-4.3431998)}] (-1.8016697,-4.3431998) ellipse (0.68803348 and 0.21599723);
\draw[thin, rotate around={-22.529916:(-1.8016697,-4.3431998)}] (-1.8016697,-4.3431998) ellipse (0.68803348 and 0.21599723);
\fill[fill=black!43.397372, rotate around={49.470084:(3.5738819,-3.0556122)}] (3.5738819,-3.0556122) ellipse (0.68803348 and 0.21599723);
\draw[thin, rotate around={49.470084:(3.5738819,-3.0556122)}] (3.5738819,-3.0556122) ellipse (0.68803348 and 0.21599723);
\fill[fill=black!42.353681, rotate around={56.979115:(-3.9007785,2.5352172)}] (-3.9007785,2.5352172) ellipse (0.68803348 and 0.23590059);
\draw[thin, rotate around={56.979115:(-3.9007785,2.5352172)}] (-3.9007785,2.5352172) ellipse (0.68803348 and 0.23590059);
\fill[fill=black!42.353681, rotate around={-87.020885:(4.6459594,0.24178626)}] (4.6459594,0.24178626) ellipse (0.68803348 and 0.23590059);
\draw[thin, rotate around={-87.020885:(4.6459594,0.24178626)}] (4.6459594,0.24178626) ellipse (0.68803348 and 0.23590059);
\fill[fill=black!42.353681, rotate around={20.979115:(1.6656328,-4.3438539)}] (1.6656328,-4.3438539) ellipse (0.68803348 and 0.23590059);
\draw[thin, rotate around={20.979115:(1.6656328,-4.3438539)}] (1.6656328,-4.3438539) ellipse (0.68803348 and 0.23590059);
\fill[fill=black!42.353681, rotate around={-51.020885:(-3.6165417,-2.9264356)}] (-3.6165417,-2.9264356) ellipse (0.68803348 and 0.23590059);
\draw[thin, rotate around={-51.020885:(-3.6165417,-2.9264356)}] (-3.6165417,-2.9264356) ellipse (0.68803348 and 0.23590059);
\fill[fill=black!42.353681, rotate around={-15.020885:(1.2057280,4.4932860)}] (1.2057280,4.4932860) ellipse (0.68803348 and 0.23590059);
\draw[thin, rotate around={-15.020885:(1.2057280,4.4932860)}] (1.2057280,4.4932860) ellipse (0.68803348 and 0.23590059);
\fill[fill=black!40.319193, rotate around={-6.5214006:(-0.51586252,-4.5126808)}] (-0.51586252,-4.5126808) ellipse (0.68803348 and 0.27422804);
\draw[thin, rotate around={-6.5214006:(-0.51586252,-4.5126808)}] (-0.51586252,-4.5126808) ellipse (0.68803348 and 0.27422804);
\fill[fill=black!40.319193, rotate around={29.478599:(-2.2351457,3.9540518)}] (-2.2351457,3.9540518) ellipse (0.68803348 and 0.27422804);
\draw[thin, rotate around={29.478599:(-2.2351457,3.9540518)}] (-2.2351457,3.9540518) ellipse (0.68803348 and 0.27422804);
\fill[fill=black!40.319193, rotate around={-78.521401:(-4.4512248,-0.90388065)}] (-4.4512248,-0.90388065) ellipse (0.68803348 and 0.27422804);
\draw[thin, rotate around={-78.521401:(-4.4512248,-0.90388065)}] (-4.4512248,-0.90388065) ellipse (0.68803348 and 0.27422804);
\fill[fill=black!40.319193, rotate around={-42.521401:(3.0698288,3.3476191)}] (3.0698288,3.3476191) ellipse (0.68803348 and 0.27422804);
\draw[thin, rotate around={-42.521401:(3.0698288,3.3476191)}] (3.0698288,3.3476191) ellipse (0.68803348 and 0.27422804);
\fill[fill=black!40.319193, rotate around={65.478599:(4.1324042,-1.8851095)}] (4.1324042,-1.8851095) ellipse (0.68803348 and 0.27422804);
\draw[thin, rotate around={65.478599:(4.1324042,-1.8851095)}] (4.1324042,-1.8851095) ellipse (0.68803348 and 0.27422804);
\fill[fill=black!37.625501, rotate around={71.926257:(-4.1534283,1.3554432)}] (-4.1534283,1.3554432) ellipse (0.68803348 and 0.32398662);
\draw[thin, rotate around={71.926257:(-4.1534283,1.3554432)}] (-4.1534283,1.3554432) ellipse (0.68803348 and 0.32398662);
\fill[fill=black!37.625501, rotate around={-36.073743:(-2.5725830,-3.5312900)}] (-2.5725830,-3.5312900) ellipse (0.68803348 and 0.32398662);
\draw[thin, rotate around={-36.073743:(-2.5725830,-3.5312900)}] (-2.5725830,-3.5312900) ellipse (0.68803348 and 0.32398662);
\fill[fill=black!37.625501, rotate around={-0.073743004:(0.0056231605,4.3690000)}] (0.0056231605,4.3690000) ellipse (0.68803348 and 0.32398662);
\draw[thin, rotate around={-0.073743004:(0.0056231605,4.3690000)}] (0.0056231605,4.3690000) ellipse (0.68803348 and 0.32398662);
\fill[fill=black!37.625501, rotate around={35.926257:(2.5634845,-3.5379005)}] (2.5634845,-3.5379005) ellipse (0.68803348 and 0.32398662);
\draw[thin, rotate around={35.926257:(2.5634845,-3.5379005)}] (2.5634845,-3.5379005) ellipse (0.68803348 and 0.32398662);
\fill[fill=black!37.625501, rotate around={-72.073743:(4.1569036,1.3447473)}] (4.1569036,1.3447473) ellipse (0.68803348 and 0.32398662);
\draw[thin, rotate around={-72.073743:(4.1569036,1.3447473)}] (4.1569036,1.3447473) ellipse (0.68803348 and 0.32398662);
\fill[fill=black!33.802712, rotate around={88.760048:(-4.0662236,0.088011893)}] (-4.0662236,0.088011893) ellipse (0.68803348 and 0.39257246);
\draw[thin, rotate around={88.760048:(-4.0662236,0.088011893)}] (-4.0662236,0.088011893) ellipse (0.68803348 and 0.39257246);
\fill[fill=black!33.802712, rotate around={-19.239952:(-1.3402365,-3.8400113)}] (-1.3402365,-3.8400113) ellipse (0.68803348 and 0.39257246);
\draw[thin, rotate around={-19.239952:(-1.3402365,-3.8400113)}] (-1.3402365,-3.8400113) ellipse (0.68803348 and 0.39257246);
\fill[fill=black!33.802712, rotate around={16.760048:(-1.1728279,3.8944056)}] (-1.1728279,3.8944056) ellipse (0.68803348 and 0.39257246);
\draw[thin, rotate around={16.760048:(-1.1728279,3.8944056)}] (-1.1728279,3.8944056) ellipse (0.68803348 and 0.39257246);
\fill[fill=black!33.802712, rotate around={52.760048:(3.2379119,-2.4612694)}] (3.2379119,-2.4612694) ellipse (0.68803348 and 0.39257246);
\draw[thin, rotate around={52.760048:(3.2379119,-2.4612694)}] (3.2379119,-2.4612694) ellipse (0.68803348 and 0.39257246);
\fill[fill=black!33.802712, rotate around={-55.239952:(3.3413761,2.3188631)}] (3.3413761,2.3188631) ellipse (0.68803348 and 0.39257246);
\draw[thin, rotate around={-55.239952:(3.3413761,2.3188631)}] (3.3413761,2.3188631) ellipse (0.68803348 and 0.39257246);
\fill[fill=black!33.225452, rotate around={-88.381231:(4.0138778,0.11343365)}] (4.0138778,0.11343365) ellipse (0.68803348 and 0.40271314);
\draw[thin, rotate around={-88.381231:(4.0138778,0.11343365)}] (4.0138778,0.11343365) ellipse (0.68803348 and 0.40271314);
\fill[fill=black!33.225452, rotate around={55.618769:(-3.3139699,2.2675284)}] (-3.3139699,2.2675284) ellipse (0.68803348 and 0.40271314);
\draw[thin, rotate around={55.618769:(-3.3139699,2.2675284)}] (-3.3139699,2.2675284) ellipse (0.68803348 and 0.40271314);
\fill[fill=black!33.225452, rotate around={19.618769:(1.3482383,-3.7823717)}] (1.3482383,-3.7823717) ellipse (0.68803348 and 0.40271314);
\draw[thin, rotate around={19.618769:(1.3482383,-3.7823717)}] (1.3482383,-3.7823717) ellipse (0.68803348 and 0.40271314);
\fill[fill=black!33.225452, rotate around={-52.381231:(-3.1806207,-2.4510679)}] (-3.1806207,-2.4510679) ellipse (0.68803348 and 0.40271314);
\draw[thin, rotate around={-52.381231:(-3.1806207,-2.4510679)}] (-3.1806207,-2.4510679) ellipse (0.68803348 and 0.40271314);
\fill[fill=black!33.225452, rotate around={-16.381231:(1.1324746,3.8524775)}] (1.1324746,3.8524775) ellipse (0.68803348 and 0.40271314);
\draw[thin, rotate around={-16.381231:(1.1324746,3.8524775)}] (1.1324746,3.8524775) ellipse (0.68803348 and 0.40271314);
\fill[fill=black!31.863748, rotate around={36:(-2.2845523,3.1444164)}] (-2.2845523,3.1444164) ellipse (0.68803348 and 0.42640081);
\draw[thin, rotate around={36:(-2.2845523,3.1444164)}] (-2.2845523,3.1444164) ellipse (0.68803348 and 0.42640081);
\fill[fill=black!31.863748, rotate around={0:(0,-3.8867125)}] (0,-3.8867125) ellipse (0.68803348 and 0.42640081);
\draw[thin, rotate around={0:(0,-3.8867125)}] (0,-3.8867125) ellipse (0.68803348 and 0.42640081);
\fill[fill=black!31.863748, rotate around={-36:(2.2845523,3.1444164)}] (2.2845523,3.1444164) ellipse (0.68803348 and 0.42640081);
\draw[thin, rotate around={-36:(2.2845523,3.1444164)}] (2.2845523,3.1444164) ellipse (0.68803348 and 0.42640081);
\fill[fill=black!31.863748, rotate around={-72:(-3.6964832,-1.2010602)}] (-3.6964832,-1.2010602) ellipse (0.68803348 and 0.42640081);
\draw[thin, rotate around={-72:(-3.6964832,-1.2010602)}] (-3.6964832,-1.2010602) ellipse (0.68803348 and 0.42640081);
\fill[fill=black!31.863748, rotate around={72:(3.6964832,-1.2010602)}] (3.6964832,-1.2010602) ellipse (0.68803348 and 0.42640081);
\draw[thin, rotate around={72:(3.6964832,-1.2010602)}] (3.6964832,-1.2010602) ellipse (0.68803348 and 0.42640081);
\fill[fill=black!28.403363, rotate around={36.761428:(2.1020513,-2.8138173)}] (2.1020513,-2.8138173) ellipse (0.68803348 and 0.48506528);
\draw[thin, rotate around={36.761428:(2.1020513,-2.8138173)}] (2.1020513,-2.8138173) ellipse (0.68803348 and 0.48506528);
\fill[fill=black!28.403363, rotate around={0.76142780:(-0.046674943,3.5119808)}] (-0.046674943,3.5119808) ellipse (0.68803348 and 0.48506528);
\draw[thin, rotate around={0.76142780:(-0.046674943,3.5119808)}] (-0.046674943,3.5119808) ellipse (0.68803348 and 0.48506528);
\fill[fill=black!28.403363, rotate around={-35.238572:(-2.0265297,-2.8686870)}] (-2.0265297,-2.8686870) ellipse (0.68803348 and 0.48506528);
\draw[thin, rotate around={-35.238572:(-2.0265297,-2.8686870)}] (-2.0265297,-2.8686870) ellipse (0.68803348 and 0.48506528);
\fill[fill=black!28.403363, rotate around={72.761428:(-3.3545155,1.0408712)}] (-3.3545155,1.0408712) ellipse (0.68803348 and 0.48506528);
\draw[thin, rotate around={72.761428:(-3.3545155,1.0408712)}] (-3.3545155,1.0408712) ellipse (0.68803348 and 0.48506528);
\fill[fill=black!28.403363, rotate around={-71.238572:(3.3256688,1.1296523)}] (3.3256688,1.1296523) ellipse (0.68803348 and 0.48506528);
\draw[thin, rotate around={-71.238572:(3.3256688,1.1296523)}] (3.3256688,1.1296523) ellipse (0.68803348 and 0.48506528);
\fill[fill=black!25.363876, rotate around={22.618402:(-1.1987526,2.8772131)}] (-1.1987526,2.8772131) ellipse (0.68803348 and 0.53467100);
\draw[thin, rotate around={22.618402:(-1.1987526,2.8772131)}] (-1.1987526,2.8772131) ellipse (0.68803348 and 0.53467100);
\fill[fill=black!25.363876, rotate around={-13.381598:(-0.72137220,-3.0323234)}] (-0.72137220,-3.0323234) ellipse (0.68803348 and 0.53467100);
\draw[thin, rotate around={-13.381598:(-0.72137220,-3.0323234)}] (-0.72137220,-3.0323234) ellipse (0.68803348 and 0.53467100);
\fill[fill=black!25.363876, rotate around={49.381598:(-2.3659573,2.0291892)}] (-2.3659573,2.0291892) ellipse (0.68803348 and 0.53467100);
\draw[thin, rotate around={49.381598:(-2.3659573,2.0291892)}] (-2.3659573,2.0291892) ellipse (0.68803348 and 0.53467100);
\fill[fill=black!25.363876, rotate around={-49.381598:(2.3659573,2.0291892)}] (2.3659573,2.0291892) ellipse (0.68803348 and 0.53467100);
\draw[thin, rotate around={-49.381598:(2.3659573,2.0291892)}] (2.3659573,2.0291892) ellipse (0.68803348 and 0.53467100);
\fill[fill=black!25.363876, rotate around={13.381598:(0.72137220,-3.0323234)}] (0.72137220,-3.0323234) ellipse (0.68803348 and 0.53467100);
\draw[thin, rotate around={13.381598:(0.72137220,-3.0323234)}] (0.72137220,-3.0323234) ellipse (0.68803348 and 0.53467100);
\fill[fill=black!25.363876, rotate around={85.381598:(3.1068272,-0.25097372)}] (3.1068272,-0.25097372) ellipse (0.68803348 and 0.53467100);
\draw[thin, rotate around={85.381598:(3.1068272,-0.25097372)}] (3.1068272,-0.25097372) ellipse (0.68803348 and 0.53467100);
\fill[fill=black!25.363876, rotate around={-22.618402:(1.1987526,2.8772131)}] (1.1987526,2.8772131) ellipse (0.68803348 and 0.53467100);
\draw[thin, rotate around={-22.618402:(1.1987526,2.8772131)}] (1.1987526,2.8772131) ellipse (0.68803348 and 0.53467100);
\fill[fill=black!25.363876, rotate around={-58.618402:(-2.6609946,-1.6231052)}] (-2.6609946,-1.6231052) ellipse (0.68803348 and 0.53467100);
\draw[thin, rotate around={-58.618402:(-2.6609946,-1.6231052)}] (-2.6609946,-1.6231052) ellipse (0.68803348 and 0.53467100);
\fill[fill=black!25.363876, rotate around={58.618402:(2.6609946,-1.6231052)}] (2.6609946,-1.6231052) ellipse (0.68803348 and 0.53467100);
\draw[thin, rotate around={58.618402:(2.6609946,-1.6231052)}] (2.6609946,-1.6231052) ellipse (0.68803348 and 0.53467100);
\fill[fill=black!25.363876, rotate around={-85.381598:(-3.1068272,-0.25097372)}] (-3.1068272,-0.25097372) ellipse (0.68803348 and 0.53467100);
\draw[thin, rotate around={-85.381598:(-3.1068272,-0.25097372)}] (-3.1068272,-0.25097372) ellipse (0.68803348 and 0.53467100);
\fill[fill=black!21.184238, rotate around={36:(1.4271437,-1.9642948)}] (1.4271437,-1.9642948) ellipse (0.68803348 and 0.59967197);
\draw[thin, rotate around={36:(1.4271437,-1.9642948)}] (1.4271437,-1.9642948) ellipse (0.68803348 and 0.59967197);
\fill[fill=black!21.184238, rotate around={0:(0,2.4280020)}] (0,2.4280020) ellipse (0.68803348 and 0.59967197);
\draw[thin, rotate around={0:(0,2.4280020)}] (0,2.4280020) ellipse (0.68803348 and 0.59967197);
\fill[fill=black!21.184238, rotate around={-36:(-1.4271437,-1.9642948)}] (-1.4271437,-1.9642948) ellipse (0.68803348 and 0.59967197);
\draw[thin, rotate around={-36:(-1.4271437,-1.9642948)}] (-1.4271437,-1.9642948) ellipse (0.68803348 and 0.59967197);
\fill[fill=black!21.184238, rotate around={72:(-2.3091671,0.75029387)}] (-2.3091671,0.75029387) ellipse (0.68803348 and 0.59967197);
\draw[thin, rotate around={72:(-2.3091671,0.75029387)}] (-2.3091671,0.75029387) ellipse (0.68803348 and 0.59967197);
\fill[fill=black!21.184238, rotate around={-72:(2.3091671,0.75029387)}] (2.3091671,0.75029387) ellipse (0.68803348 and 0.59967197);
\draw[thin, rotate around={-72:(2.3091671,0.75029387)}] (2.3091671,0.75029387) ellipse (0.68803348 and 0.59967197);
\fill[fill=black!19.498314, rotate around={36:(-1.2195028,1.6785017)}] (-1.2195028,1.6785017) ellipse (0.68803348 and 0.62474596);
\draw[thin, rotate around={36:(-1.2195028,1.6785017)}] (-1.2195028,1.6785017) ellipse (0.68803348 and 0.62474596);
\fill[fill=black!19.498314, rotate around={0:(0,-2.0747422)}] (0,-2.0747422) ellipse (0.68803348 and 0.62474596);
\draw[thin, rotate around={0:(0,-2.0747422)}] (0,-2.0747422) ellipse (0.68803348 and 0.62474596);
\fill[fill=black!19.498314, rotate around={-36:(1.2195028,1.6785017)}] (1.2195028,1.6785017) ellipse (0.68803348 and 0.62474596);
\draw[thin, rotate around={-36:(1.2195028,1.6785017)}] (1.2195028,1.6785017) ellipse (0.68803348 and 0.62474596);
\fill[fill=black!19.498314, rotate around={72:(1.9731971,-0.64113059)}] (1.9731971,-0.64113059) ellipse (0.68803348 and 0.62474596);
\draw[thin, rotate around={72:(1.9731971,-0.64113059)}] (1.9731971,-0.64113059) ellipse (0.68803348 and 0.62474596);
\fill[fill=black!19.498314, rotate around={-72:(-1.9731971,-0.64113059)}] (-1.9731971,-0.64113059) ellipse (0.68803348 and 0.62474596);
\draw[thin, rotate around={-72:(-1.9731971,-0.64113059)}] (-1.9731971,-0.64113059) ellipse (0.68803348 and 0.62474596);
\fill[fill=black!16.401532, rotate around={36:(0.68148819,-0.93798802)}] (0.68148819,-0.93798802) ellipse (0.68803348 and 0.66891303);
\draw[thin, rotate around={36:(0.68148819,-0.93798802)}] (0.68148819,-0.93798802) ellipse (0.68803348 and 0.66891303);
\fill[fill=black!16.401532, rotate around={0:(0,1.1594170)}] (0,1.1594170) ellipse (0.68803348 and 0.66891303);
\draw[thin, rotate around={0:(0,1.1594170)}] (0,1.1594170) ellipse (0.68803348 and 0.66891303);
\fill[fill=black!16.401532, rotate around={72:(-1.1026710,0.35827954)}] (-1.1026710,0.35827954) ellipse (0.68803348 and 0.66891303);
\draw[thin, rotate around={72:(-1.1026710,0.35827954)}] (-1.1026710,0.35827954) ellipse (0.68803348 and 0.66891303);
\fill[fill=black!16.401532, rotate around={-36:(-0.68148819,-0.93798802)}] (-0.68148819,-0.93798802) ellipse (0.68803348 and 0.66891303);
\draw[thin, rotate around={-36:(-0.68148819,-0.93798802)}] (-0.68148819,-0.93798802) ellipse (0.68803348 and 0.66891303);
\fill[fill=black!16.401532, rotate around={-72:(1.1026710,0.35827954)}] (1.1026710,0.35827954) ellipse (0.68803348 and 0.66891303);
\draw[thin, rotate around={-72:(1.1026710,0.35827954)}] (1.1026710,0.35827954) ellipse (0.68803348 and 0.66891303);
\end{tikzpicture}
\caption{This packing of $180$ spherical caps on a transparent sphere corresponds to a conjecturally optimal spherical code in $\R^3$ with chiral icosahedral symmetry \cite{T83}. Its maximal inner product is the root $0.9621287940857901\dots$ of the polynomial $448505x^{18} + 1214430x^{17} + 1867769x^{16} + 1054048x^{15} - 1540076x^{14}
- 3472264x^{13} - 2273628x^{12} + 709568x^{11} + 2115198x^{10}
+ 712036x^9 - 369522x^8 - 221472x^7 + 20772x^6 + 3416x^5
- 6956x^4 + 256x^3 + 81x^2 - 18x + 1$.}
\label{figure:plot}
\end{figure}

Nevertheless, rigorous proofs of optimality are known in certain cases, which are listed in Table~\ref{table:knowncases}. When $n=3$, the cases with $N \le 14$ or $N=24$ can be resolved by direct arguments based on geometry and combinatorics,\footnote{For comparison, the optimal cosine for $15$ points on $S^2$ is conjectured to be $0.59260590\dots$, with minimal polynomial $13x^5 - x^4 + 6x^3 + 2x^2 - 3x - 1$. There appear to be exactly two optimal codes, up to isometry: the first code to be discovered \cite{SW51} and a modification that moves one point \cite{K91}.} and the same is true for $120$ points in $\R^4$. However, these direct techniques have not yet had any success in other cases.

\begin{table}
\caption{The previously known optimal spherical codes with $N>2n$ points in $\R^n$ and minimal angle $\theta$, for $n>2$.}
\label{table:knowncases}
\begin{tabular}{rrlll}
\toprule
$n$ & $N$ & $\cos \theta$ & References\\
\midrule
$3$ & $7$ & root of $3x^3 - 9x^2 - 3x + 1$ near $0.210$ & \cite{SW51}\\
$3$ & $8$ & $1/\big(\sqrt{8}+1\big)$ & \cite{SW51}\\
$3$ & $9$ & $1/3$ & \cite{SW51}\\
$3$ & $10$ & root of $7x^3 - 4x^2 - 2x + 1$ near $0.404$ & \cite{SW51,D86}\\
$3$ & $11$ & $1/\sqrt{5}$ & \cite{SW51,D86}\\
$3$ & $12$ & $1/\sqrt{5}$ & \cite{FT43}\\
$3$ & $13$ & root of $x^8 - 24x^7 - 12x^6 + 8x^5 + 38x^4$ & \cite{SW51,MT12}\\
 & & $\phantom{} + 24x^3 - 12x^2 - 8x + 1$ near $0.543$\\
$3$ & $14$ & root of $4x^4 - 2x^3 + 3x^2 - 1$ near $0.564$ & \cite{SW51,MT15}\\
$3$ & $24$ & root of $7x^3 + x^2 - 3x - 1$ near $0.723$ & \cite{R61}\\
$4$ & $10$ & $1/6$ & \cite{M80,BV09}\\
$4$ & $120$ & $1/\big(\sqrt{5}-1\big)$ & \cite{S01,B78,A99}\\
$5$ & $16$ & $1/5$ & \cite{G00,L79}\\
$6$ & $27$ & $1/4$ & \cite{G00,L79}\\
$7$ & $56$ & $1/3$ & \cite{G00,L79}\\
$8$ & $240$ & $1/2$ & \cite{KZ73,L79,OS79}\\
$21$ & $112$ & $1/9$ & \cite{DGS77,L79}\\
$21$ & $162$ & $1/7$ & \cite{DGS77,L79}\\
$22$ & $100$ & $1/11$ & \cite{DGS77,L79}\\
$22$ & $275$ & $1/6$ & \cite{DGS77,L79}\\
$22$ & $891$ & $1/4$ & \cite{DGS77,L79}\\
$23$ & $552$ & $1/5$ & \cite{DGS77,L79}\\
$23$ & $4600$ & $1/3$ & \cite{DGS77,L79}\\
$24$ & $196560$ & $1/2$ & \cite{L67,L79,OS79}\\
$q\frac{q^3+1}{q+1}$ & $(q+1)(q^3+1)$ & $1/q^2$ where $q$ is a prime power & \cite{CGS78,L87,L92}\\
\bottomrule
\end{tabular}
\end{table}

All the remaining cases are proved using the theory of linear or semidefinite programming bounds, developed in \cite{D72, DGS77, KL78, L92, S05, BV08, M07, LV15} and many other papers. Specifically, \emph{$k$-point bounds} are based on certain constraints on the $k$-point distribution of a code $\mathcal{C}$, which counts $k$-tuples of points according to their relative positions. For $k=2$ there are linear constraints on this distribution, and for $k \ge 3$ there are semidefinite constraints. The $k$-point bound maximizes the number of points for a given minimal angle subject only to these constraints.\footnote{That problem is different from maximizing the minimal angle given the number of points, but the conditions needed for equality turn out to imply optimality as a spherical code in each of these sharp cases.} There is no reason to expect these specific constraints to be comprehensive, and indeed the $k$-point bound usually seems to be worse than the exact optimum. For example, for a spherical code with $24$ points in $\R^4$, the two-point bound implies only that the optimal cosine can be no smaller than $0.49251502\dots$ (as shown in \cite{AB00}), which does not match the conjectured value of $1/2$. However, it always provides a bound on the optimum, and in certain remarkable cases the bound is unexpectedly sharp (namely, the remaining cases in Table~\ref{table:knowncases}). This phenomenon is a compact analogue of the solution of the sphere packing problem in $8$ and $24$ dimensions \cite{V17,CKMRV17}.

Among the previously known cases, only one required three-point bounds\footnote{Three-point bounds also prove optimality for $8$ points in $\R^3$, as shown in \cite{DLM21}, but that case can be proved directly.} (the proof by Bachoc and Vallentin \cite{BV09} for $10$ points in $\R^4$), while all the others used only two-point bounds. Cohn and Woo \cite[Section~5]{CW12} raised the question of why three-point bounds were not known to be sharp more often: it seemed strange that two-point bounds could prove optimality for many spherical codes, while three-point bounds were known to be sharp in only a few additional cases. In the present paper, we show that three-point bounds are sharp in substantially more cases than was previously believed. Specifically, we use three-point bounds to prove the following theorem:

\begin{table}
\caption{The spherical codes we prove are optimal. Each consists of $N$ points in $\R^n$, with minimal angle $\theta$.}
\label{table:newcases}
\begin{tabular}{rrlll}
\toprule
$n$ & $N$ & $\cos \theta$ & Other cosines & Name and references\\
\midrule
$16$ & $288$ & $1/4$ & $-1$, $-1/4$, $0$ & \makecell[l]{Nordstrom-Robinson spherical code\\ \cite{NR68, CS73, L82, K95}}\\
$20$ & $56$ & $1/15$ & $-2/5$ & Gewirtz graph \cite{G69a,EZ01}\\
$21$ & $50$ & $1/21$ & $-3/7$ & Hoffman-Singleton graph \cite{HS60,EZ01}\\
$21$ & $77$ & $1/12$ & $-3/8$ & $M_{22}$ graph \cite[p.~75]{M64}, \cite{EZ01}\\
$64$ & $4224$ & $1/8$ & $-1$, $-1/8$, $0$ & Kerdock spherical code \cite{K72, CS73, L82, K95}\\
$256$ & $66048$ & $1/16$ & $-1$, $-1/16$, $0$ & Kerdock spherical code \cite{K72, CS73, L82, K95}\\
$1024$ & $1050624$ & $1/32$ & $-1$, $-1/32$, $0$ & Kerdock spherical code \cite{K72, CS73, L82, K95}\\
\bottomrule
\end{tabular}
\end{table}

\begin{theorem} \label{theorem:main}
Each of the spherical codes listed in Table~\ref{table:newcases} is optimal, and those in at most $64$ dimensions are unique up to isometry. Furthermore, the optimal $288$-point spherical code in $\R^{16}$ is universally optimal.
\end{theorem}

These results are proved using rigorous computer calculations, for which data files and code are available \cite{CdLL24}. In each case, two-point bounds are not sharp, as shown in Appendix~\ref{app:two-point}. The optimal spherical codes fall into two families, which are derived from triangle-free strongly regular graphs and Kerdock codes.

\subsection{Triangle-free strongly regular graphs}

By a graph, we mean an undirected, simple graph (no multiple edges or self-loops). Recall that a \emph{strongly regular graph with parameters $(n,k,\lambda,\mu)$} is an $n$-vertex graph, not a complete graph or its complement, such that every vertex has degree $k$, every pair of adjacent vertices has $\lambda$ common neighbors, and every pair of distinct, non-adjacent vertices has $\mu$ common neighbors. One can check that such a graph is connected if and only if $\mu > 0$, and we will always assume this is the case.

A graph is \emph{triangle-free} if there do not exist three mutually adjacent vertices; for a strongly regular graph, this condition amounts to $\lambda=0$. The only known connected, triangle-free strongly regular graphs are one infinite family and seven exceptional cases, namely
\begin{enumerate}
\item the complete bipartite graph $K_{n,n}$, which has parameters $(2n,n,0,n)$,

\item the $5$-cycle, which has parameters $(5,2,0,1)$,

\item the Petersen graph, which has parameters $(10,3,0,1)$,

\item the Clebsch graph, which has parameters $(16,5,0,2)$,

\item the Hoffman-Singleton graph, which has parameters $(50,7,0,1)$,

\item the Gewirtz graph, which has parameters $(56,10,0,2)$,

\item the $M_{22}$ graph, which has parameters $(77,16,0,4)$, and

\item the Higman-Sims\footnote{Before it was rediscovered by Higman and Sims \cite{HS68} in the process of constructing the Higman-Sims group, this graph was discovered by Mesner \cite{M56}, as pointed out by Klin and Woldar \cite{KW17}.} graph, which has parameters $(100,22,0,6)$.
\end{enumerate}
See \cite{BvM22} for constructions of these graphs, as well as more information and references.

Each of these graphs corresponds to a spherical code via \emph{spectral embedding}. Specifically, let $A$ be the graph's adjacency matrix, which is indexed by vertices. Orthogonally projecting the basis vectors indexed by these vertices into an eigenspace of $A$ yields points on a sphere, which we can rescale to be the unit sphere.

The complete bipartite graph $K_{n,n}$ is a special case, because it is bipartite. The eigenvalues of its adjacency matrix are $0$ (with multiplicity $2n-2$) and $\pm n$, and projecting into the eigenspace with eigenvalue~$0$ gives two orthogonal $(n-1)$-dimensional regular simplices in $\R^{2n-2}$, which is an optimal spherical code, but not unique for $n>2$.

For each of the remaining cases, we orthogonally project into the eigenspace with the smallest eigenvalue. The resulting spherical code is a two-distance set (i.e., there are only two distances between distinct points), with the larger distance corresponding to adjacency in the graph. Each of these codes is either already known to be optimal (Table~\ref{table:knowncases}\footnote{The $5$-cycle corresponds to a regular pentagon in the unit circle, which is too small to be listed in Table~\ref{table:knowncases}.}) or proved to be optimal in this paper (Table~\ref{table:newcases}). The new cases are the Hoffman-Singleton graph, the Gewirtz graph, and the $M_{22}$ graph.

Each of these three cases is closely related to the previously known optimal spherical codes. The $M_{22}$ spherical code is the kissing configuration of the Higman-Sims spherical code; in other words, it can be obtained as the nearest neighbors of any point in the Higman-Sims code. The Hoffman-Singleton and Gewirtz codes do not arise in this way; instead they occur as facets of other codes. Specifically, the optimal code with $112$ points in $\R^{21}$ is an antiprism consisting of two copies of the Gewirtz code on parallel planes, and the Higman-Sims code is an antiprism consisting of two copies of the Hoffman-Singleton code on parallel planes. For the combinatorics of this construction, see \cite[Section~5(d)]{BH93} for the Gewirtz graph and \cite[\S10.31]{BvM22} for the Hoffman-Singleton graph; for the geometric interpretation as an antiprism, see \cite[Section~3.7]{BBCGKS09}.

\begin{conjecture} \label{conj:stronglyregular}
Let $G$ be a connected, triangle-free strongly regular graph other than a complete bipartite graph, and let $\mathcal{C}$ be the spectral embedding of $G$ into its eigenspace with the smallest eigenvalue. Then three-point bounds prove that $\mathcal{C}$ is an optimal spherical code.
\end{conjecture}

Conjecture~\ref{conj:stronglyregular} may not have further implications, since it is plausible that no more triangle-free strongly regular graphs remain to be discovered. However, our proofs work by examining each case separately. Even if no additional cases are found, it would be preferable to have a unified, more conceptual proof of optimality.

While we do not know how to prove Conjecture~\ref{conj:stronglyregular}, we prove optimality for several hypothetical codes based on possible parameters for triangle-free strongly regular graphs. These parameters are shown in Table~\ref{table:hypotheticalcases} and were taken from the list of possible strongly regular graph parameters in Chapter~12 of \cite{BvM22} (together with the Moore graph of degree~$57$, which has parameters $(3250,57,0,1)$). We imagine that few, if any, of the codes listed in this table actually exist, but if they do, then they are optimal.

\begin{table}
\caption{Hypothetical spherical codes that are optimal if they exist. Each consists of $N$ points in $\R^n$, with angles $\theta$ and $\psi$ between distinct points.}
\label{table:hypotheticalcases}
\begin{tabular}{rrlll}
\toprule
$n$ & $N$ & $\cos \theta$ & $\cos \psi$ & Parameters\\
\midrule
$55$ & $176$ & $1/25$  & $-7/25$ & $(176,25,0,4)$\\
$55$ & $210$ & $1/22$ & $-3/11$ & $(210,33,0,6)$\\
$56$ & $162$ & $1/28$ & $-2/7$ & $(162,21,0,3)$\\
$56$ & $266$ & $1/20$ & $-4/15$ & $(266,45,0,9)$\\
$115$ & $392$ & $3/115$ & $-5/23$ & $(392,46,0,6)$\\
$120$ & $352$ & $1/45$ & $-2/9$ & $(352,36,0,4)$\\
$143$ & $352$ & $1/65$ & $-3/13$ & $(352,26,0,2)$\\
$1520$ & $3250$ & $1/456$ & $-8/57$ & $(3250,57,0,1)$\\
\bottomrule
\end{tabular}
\end{table}

\subsection{Kerdock codes}
\label{subsec:kerdock}

Kerdock codes \cite{K72} are a family of binary error-correcting codes in $\{0,1\}^{2^{2k}}$ with $2^{4k}$ codewords and minimal Hamming distance $2^{2k-1}-2^{k-1}$, which can be constructed using $\Z/4\Z$-linear codes \cite{HKCSS95}. The case $k=1$ is trivial, while $k=2$ is the Nordstrom-Robinson code \cite{NR68}, which is known to be not only optimal \cite{NR68} but also unique \cite{S73}. For $k > 2$ it has not been known whether the Kerdock codes are optimal.

For each $k$, one can use the Kerdock code to construct a spherical code with $2^{4k}+2^{2k+1}$ points in $2^{2k}$ dimensions and maximal inner product $1/2^k$ (see \cite{CS73, L82, K95}). Specifically, this spherical code contains the standard orthonormal basis and its negatives as well as the $2^{4k}$ points  $((-1)^{c_1},\dots,(-1)^{c_{2^{2k}}})/2^k$ with $(c_1,\dots,c_{2^{2k}})$ in the Kerdock code. We call these configurations the \emph{Kerdock spherical codes}. They are known to be optimal among antipodal spherical codes \cite{L82}, or equivalently as point configurations in real projective space, but have not previously been known to be optimal among all spherical codes.

The points in Kerdock spherical codes can be grouped to form $2^{2k-1}+1$ cross polytopes (recall that a regular \emph{cross polytope} in $\R^n$ consists of $n$ orthogonal pairs of antipodal unit vectors), with inner products $\pm 1/2^k$ between points in different cross polytopes. In other words, they come from $2^{2k-1}+1$ real mutually unbiased bases, which is the maximal possible number in $2^{2k}$ dimensions (see Table~I in \cite{DGS75} or equation~(3.9) in \cite{CCKS97}). Mutually unbiased bases play a fundamental role in quantum information theory \cite{WF89}.

As stated in Theorem~\ref{theorem:main} and Table~\ref{table:newcases}, we prove that Kerdock spherical codes are optimal for $2 \le k \le 5$. Furthermore, we prove that they are unique for $k=2$ and $k=3$. As a corollary of our theorem, we obtain optimality of the Kerdock binary codes for $3 \le k \le 5$ and uniqueness for $k=3$, which were open problems (the Nordstrom-Robinson case $k=2$ was previously known). Furthermore, restricting from the sphere to the vertices of an inscribed hypercube shows that Schrijver's three-point bound for binary codes \cite{S05} is sharp in these cases.

We suspect that three-point bounds are sharp more generally:

\begin{conjecture} \label{conj:kerdock}
Three-point bounds prove optimality for Kerdock spherical codes in each dimension $2^{2k}$ with $k \ge 2$, and hence also for the corresponding Kerdock binary codes.
\end{conjecture}

Strangely, three-point bounds do not seem to be sharp when $k=1$. In that case, the spherical code consists of the vertices of a regular $24$-cell, or equivalently the $D_4$ root system. It solves the kissing problem in four dimensions \cite{M08}, but it is not known to be an optimal spherical code. In many ways $k=1$ is the most appealing case, and we have no conceptual explanation for why it seems more difficult to resolve than the others.

Our techniques are likely capable of resolving Conjecture~\ref{conj:kerdock} for any specific choice of $k \ge 2$ for which it holds, but the semidefinite programs that must be solved seem to grow as $k$ grows. Perhaps one can identify a pattern and settle all these cases systematically, but we have not succeeded in doing so. In the notation of Section~\ref{sec:computation}, the required degree is $d=8$ for $k=2$, $d=9$ for $k=3$, $d=10$ for $k=4$, and $d=18$ for $k=5$. The unexpected growth of $d$ from $k=4$ to $k=5$ suggests that finding a systematic solution may be difficult. Perhaps Conjecture~\ref{conj:kerdock} is not even true for large $k$, and four-point bounds will at some point be required.

\subsection{Other spherical codes}

Besides the codes listed in Table~\ref{table:newcases}, we use three-point bounds to give new proofs of optimality for the spherical code with $9$ points in $\R^3$. This case was previously known to be optimal \cite{SW51}, but three-point bounds were not known to be sharp. We conjecture that the same holds for $24$ points:

\begin{conjecture} \label{conjecture:snub}
Three-point bounds prove that the $24$ vertices of a snub cube are an optimal spherical code in three dimensions.
\end{conjecture}

Robinson \cite{R61} proved that they form an optimal spherical code, but semidefinite programming bounds would provide a different proof. It is likely that Conjecture~\ref{conjecture:snub} could be proved by using our rounding technique, but so far the precision required seems to exceed the capacity of our implementation.

\subsection{Universal optimality}

Theorem~\ref{theorem:main} says that the optimal $288$-point spherical code in $\R^{16}$ is not just optimal, but \emph{universally optimal} in the sense of Cohn and Kumar \cite{CK07}. In other words, it minimizes the energy
\[
\sum_{\substack{x,y \in \mathcal{C}\\ x \ne y}} f(\langle x,y \rangle)
\]
for every absolutely monotonic function $f \colon [-1,1) \to \R$. Recall that an \emph{absolutely monotonic} function $f$ is infinitely differentiable and satisfies $f^{(k)} \ge 0$ for all $k \ge 0$. If we set $g(r) = f(1-r/2)$, then $f$ is absolutely monotonic if and only if $(-1)^k g^{(k)} \ge 0$ for all $k \ge 0$ (i.e., $g$ is \emph{completely monotonic}), and the energy becomes
\[
\sum_{\substack{x,y \in \mathcal{C}\\ x \ne y}} g(|x-y|^2).
\]
In other words, universally optimal configurations minimize energy for every completely monotonic potential function of squared Euclidean distance. For example, such functions include inverse power laws.

By contrast, the spherical codes based on the Gewirtz, Hoffman-Singleton, and $M_{22}$ graphs are not universally optimal. In particular, each of these codes is suboptimal for the logarithmic energy
\[
\sum_{\substack{x,y \in \mathcal{C}\\ x \ne y}} -\log |x-y|,
\]
as shown by explicit examples in the data set \cite{CdLL24} that accompanies this paper. Because
\[
\lim_{\varepsilon \to 0+} \frac{1/r^{\varepsilon}-1}{\varepsilon} = - \log r,
\]
these codes are also suboptimal for inverse power laws $r \mapsto 1/r^\varepsilon$ for small, positive values of $\varepsilon$.

We do not know whether the Kerdock spherical codes in dimensions greater than~$16$ are universally optimal. If they are, then it seems likely that universal optimality can be proved using similar techniques. For comparison, the $D_4$ root system, which is the four-dimensional Kerdock spherical code, is not universally optimal \cite{CCEK07}. Again we have no conceptual explanation for why the cases $k=1$ and $k=2$ differ in this respect.

\subsection{Organization of the paper}

In Sections~\ref{sec:roundingproc} and~\ref{sec:computation}, we explain our computational procedure and show how to obtain the optimality results from Tables~\ref{table:newcases} and~\ref{table:hypotheticalcases}. We then complete the proof of Theorem~\ref{theorem:main} by dealing with universal optimality in Section~\ref{sec:universal} and uniqueness in Section~\ref{sec:uniqueness}, and we conclude in Section~\ref{sec:conjectures} by discussing other potential targets for proofs of optimality. Finally, we verify in Appendix~\ref{app:two-point} that two-point bounds are not sharp, and we construct several of the codes from Section~\ref{sec:conjectures} in Appendix~\ref{app:constructions}.

\section{Rounding procedure}
\label{sec:roundingproc}

Our proof of Theorem~\ref{theorem:main} is based on semidefinite programming bounds. The primary technical obstacle is how to convert the approximate, floating-point output of a semidefinite programming solver to an exact optimal solution. Dostert, de Laat, and Moustrou \cite{DLM21} proposed a rounding method for this purpose, which works for some of the smaller semidefinite programs considered in this paper but is not fast enough for the larger problems we need. In this section we describe an improved approach, which is much faster for large semidefinite programs. For example, when we apply our improved method to the largest instance considered in \cite{DLM21}, rounding takes only $30$ seconds, compared to $8$ hours as reported in their paper.

A semidefinite program in primal form is an optimization problem of the form
\begin{equation}\label{eq:primalsdp}
\begin{aligned}
& \text{minimize} & & \langle C, X \rangle\\
& \text{subject to} & & \langle A_i, X \rangle = b_i \text{ for } i=1,\dots,m \text{ and}\\
&&&X \in \Smat_+^n.
\end{aligned}
\end{equation}
Here $\Smat_+^n$ is the cone of positive semidefinite $n \times n$ matrices over $\R$ (for comparison, $\Smat^n$ will denote the set of symmetric $n \times n$ matrices, and $\Smat_{++}^n$ the cone of positive definite $n \times n$ matrices), and $\langle A, B \rangle = \mathrm{tr}(A^{\sf T} B)$ is the trace inner product. The problem is specified by the symmetric $n \times n$ matrices $C, A_1, \dots, A_m$ and real scalars $b_1, \dots, b_m$, and we generally assume the affine constraints are linearly independent.
The corresponding dual semidefinite program reads
\[
\begin{aligned}
& \text{maximize} & & \langle b,y \rangle\\
& \text{subject to} & & C - \sum_{i=1}^m y_iA_i \in \Smat_+^n.
\end{aligned}
\]

We assume the existence of strictly feasible points for both the primal and dual semidefinite programs, which are solutions where the matrix is positive definite. By Slater's criterion, this condition guarantees the existence of optimal primal and dual solutions, and these solutions have the same objective value (see, for example, \S5.3.2 in \cite{BV04}). We also assume the existence of a strictly complementary solution, which is a pair $(X, y)$ of primal and dual solutions such that 
\[
\mathrm{rank}(X) + \mathrm{rank}\mathopen{}\left(C - \sum_{i=1}^m y_iA_i\right)\mathclose{} = n.
\]
Under this assumption, the central path (which is followed by interior-point methods used in practice) converges to the analytic center of the optimal face \cite{HKR02}. This means we can use a solver to find a numerical approximation of a point in the relative interior of the optimal face.

Although there exist rational semidefinite programs for which the following is not the case \cite{N10}, for many problems that we encounter in practice the affine hull of the optimal face is given by affine equations whose coefficients lie in a common algebraic number field of low degree. The rounding method from \cite{DLM21} finds and uses such a description to round a numerical solution to an exact optimal solution over the field of algebraic numbers. In this section, we explain this heuristic, and we explain our changes that make it much faster. 

In practice, the semidefinite programs we consider are in block form (where the cone $\Smat_+^n$ is replaced by a product of positive semidefinite matrix cones), but here we only consider the case of one block, as the extension to multiple blocks is immediate. We also only consider the case of rounding to rationals and postpone the case of rounding to a number field to Section~\ref{sec:notq}. 

One simple approach is to compute the solution to high precision (using an arbitrary precision solver) and then use the LLL lattice basis reduction algorithm \cite{LLL82} to round each entry of the numerical solution $X^*$ to a nearby algebraic number. This method does work for some examples, but in our experiments, it often does not work when the optimal solution is not unique. For example, we have not been able to find the number field for the analytic center of the semidefinite program we used in the optimality proof of the $56$-point spherical code in~$20$ dimensions, even after solving the semidefinite program to extremely high precision. It seems the analytic center, in this case, requires an algebraic number field of high degree or with large coefficients.

Simply projecting the solution into the affine space given by the constraints $\langle A_i, X \rangle = b_i$ for $i=1,\dots,m$ also often does not work. The issue here is that the dimension of the optimal face of a semidefinite program is usually smaller than the dimension of this affine space. This dimension mismatch implies that if we project a numerical solution into the affine space, the projected solution will generally no longer be positive semidefinite. 

The rounding procedure we use consists of three steps. First we find a suitable description of the affine hull of the optimal face (Section~\ref{sec:optface}), then we transform the problem (Section~\ref{sec:transform}), and finally we round the numerical solution to a solution of the transformed problem (Section~\ref{sec:round}). For each of the three steps, we first describe the approach from \cite{DLM21} and then give our new approach.

\subsection{Describing the optimal face}\label{sec:optface}

Let $\mathcal{S}$ be the set of feasible solutions of the primal semidefinite program \eqref{eq:primalsdp}, and $\mathcal{F}$ the face of $\mathcal{S}$ consisting of the optimal solutions. As shown in \cite{HW87}, the minimal face of $\Smat_+^n$ containing a matrix $X \in \Smat_+^n$ is given by
\[
\{Y \in \Smat_+^n : \ker X \subseteq \ker Y\}.
\]
This implies that if $X$ lies in the relative interior of $\mathcal{F}$, then 
\[
\mathcal{F} = \{ Y \in \Smat_+^n : \langle A_i, Y \rangle = b_i \text{ for } i=1,\dots,m \text{ and } \ker X \subseteq \ker Y\} 
\]
Note that $\ker X$ is independent of the choice of $X$, as long as $X$ lies in the relative interior of $\mathcal{F}$. We assume this kernel has a nice description over $\Q$, by which we mean it is the solution set of some affine functions whose coefficients are rational numbers of reasonably low bit size (i.e., not excessively large numerators or denominators).

The rounding procedure starts by using an interior-point method to find a solution $X^*$ close to the analytic center of the optimal face. Since the analytic center lies in the relative interior of $\mathcal{F}$, there exists a basis of its kernel of small bit size. We want to use $X^*$ to find such a basis.

To do so, the following procedure is used in \cite{DLM21}. First, a basis for the kernel of the numerical solution matrix $X^*$ is computed numerically (for this the dual solution could also be used), and the basis vectors are listed as the rows of a matrix. Then the LLL algorithm is used to compute integer relations between the columns of this matrix. Once sufficiently many relations have been found, the coefficient vectors of these relations are listed as the rows of an integral matrix and a basis for the kernel of this matrix is computed in exact arithmetic. The problem with this approach is that the LLL algorithm is too time-consuming for large matrices. The computational complexity for finding one relation scales as $n^6$ when $X^*$ is of size $n \times n$.

We therefore find the kernel vectors as follows. We list the $k$ numerically computed kernel vectors as the rows of a matrix $M \in \R^{k \times n}$, and let $M'$ be its reduced row-echelon form. Since the reduced row-echelon form is unique, and since $X^*$ approximates the analytic center of the optimal face, whose kernel vectors are assumed to be definable over $\Q$, the entries of $M'$ are close approximations of rational numbers. We then replace each entry in $M'$ by the corresponding rational number.

Although this approach works in principle (given that we use high enough floating point precision) and is much faster than the corresponding procedure in \cite{DLM21}, we notice that the obtained kernel vectors can have significantly larger bit size, which becomes an issue in Section~\ref{sec:transform} and \ref{sec:round}.

We would like to use the LLL algorithm to reduce the bit size of the vectors (note that this scales as $k^5n$ instead of the $n^6$ discussed above). When applied to the rows of $M'$, the algorithm finds integer linear combinations of the rows which are closer to being orthogonal to each other and have smaller norms. To make these outcomes correspond to a smaller bit size, $M'$ first needs to be converted to an integral matrix. If we do so by clearing the denominators per row, then the pivot columns of $M'$ become columns with exactly one large nonzero entry, so that these entries cannot be reduced by the LLL algorithm. We therefore convert $M'$ to an integral matrix using a different method.

Let $N \in \Q^{n \times (n-k)}$ be a matrix of full column rank such that $M' N = 0$, which is easily obtained since $M'$ is of full row rank and in reduced row-echelon form, and let $N'$ be the matrix obtained from $N$ by clearing denominators in each column. We then compute the row Hermite normal form 
\begin{equation}\label{eq:hnf}
\begin{pmatrix} H \\ 0_{k \times n} \end{pmatrix} = \begin{pmatrix} T_H \\ T_0 \end{pmatrix} N',    
\end{equation}
where 
\[
T = \begin{pmatrix} T_H \\ T_0 \end{pmatrix}
\]
is a unimodular transformation matrix with $T_0 \in \Z^{k \times n}$. It follows that the rows of $T_0$ form a basis of the row space of $M$, and thus an integral basis of the kernel. We use this matrix as the integral version of $M'$, and we can now further improve the bit size of the rows using the LLL algorithm. 

In the implementation, we also use this reduction approach to check heuristically whether we are using enough precision to compute $X^*$ and to round the entries of $M'$. In practice, the vectors obtained after reduction will still be approximate kernel vectors of $X^*$ if and only if we used high enough precision.

\subsection{Transforming the problem}\label{sec:transform}

In \cite{DLM21}, the affine hull of $\mathcal{F}$ is described as 
\[
\mathcal{L} = \{ X \in \Smat^n : \langle A_i, X \rangle = b_i \text{ for } i=1,\dots,m \text{ and } Xv_1 = \dots = X v_k = 0 \},
\]
where $v_1, \dots, v_k$ are the kernel vectors obtained in Section~\ref{sec:optface}. To obtain an exact optimal solution, $X^*$ can be projected into $\mathcal{L}$, so that near zero eigenvalues become exactly zero and larger eigenvalues stay positive. This approach, however, increases the problem size by a large factor. For example, for the semidefinite programs needed to prove universal optimality in Section~\ref{sec:universal}, it increases the number of constraints from $710$ to $10835$.

Instead of describing the affine hull of $\mathcal{F}$ by adding kernel constraints, we can describe it using a coordinate transform; see, for example, \cite[Section 31.5]{DL97}. Let $B$ be a rational, invertible matrix for which the first $k$ rows form a basis of the kernel. Applying this basis transformation to any matrix $X$ in the relative interior of the optimal face gives a matrix of the form 
\[
 B X B^{\sf T} = \begin{pmatrix}
     0 & 0 \\
     0 & \widehat X
 \end{pmatrix},
\]
where $\widehat X$ lies in the cone $\Smat^{n-k}_{++}$ of positive definite matrices. Let $\widehat A_i$ be the block of $B^{-\sf T} A_i B^{-1}$ corresponding to $\widehat X$, so that
\[
\langle \widehat A_i, \widehat X \rangle = \langle A_i, X\rangle.
\]
With
\[
\widehat{\mathcal{L}} = \{\widehat X \in \Smat^{n-k} : \langle \widehat A_i, \widehat X \rangle = b_i \text{ for } i=1, \dots, m \},
\]
each element in the spectrahedron $\Smat_+^{n-k} \cap \widehat{\mathcal{L}}$ corresponds to an optimal solution in the original semidefinite program. We can now obtain from $X^*$ a positive definite matrix $\widehat X^*$ close to $\widehat{\mathcal{L}}$, and projecting into $\widehat{\mathcal{L}}$ gives a matrix that corresponds to an exact optimal solution. 

To find the basis transformation matrix $B$, we consider two methods. In our experiments, both methods work well, and it is not clear which of these two methods is better.

The first option is to extend the basis of the kernel with linearly independent standard basis vectors. We can find these vectors by using Gram-Schmidt orthogonalization in floating-point arithmetic. 

The second option is to extend the basis by using the rows of the matrix $T_H$ from \eqref{eq:hnf}. Then $B$ consists of the LLL reduced rows of $T_0$ and the rows of $T_H$, so that $B$ is unimodular. To obtain a matrix $B$ with small bit size it is desirable that $T_H$ should have small entries. To achieve this we use the method of  \cite{HL12}: instead of \eqref{eq:hnf}, we consider the Hermite normal form of the matrix $N'$ with an identity matrix appended to it. This decomposition amounts to 
\[
\begin{pmatrix}
    H & T_H \\
    0 & T_0
\end{pmatrix} = \begin{pmatrix} T_H \\ T_0 \end{pmatrix} \begin{pmatrix} N' & I_{n}\end{pmatrix}, 
\]
where now $T_0$ is in row Hermite normal form and $T_H$ is reduced with respect to $T_0$.

\subsection{Rounding the solution}\label{sec:round}

Given an approximate solution $x^*$ to a linear system $Ax = b$, we want to find an exact solution $\bar x$ close to $x^*$. In the rounding procedure, the linear system comes from the constraints $\langle \widehat A_i, \widehat X \rangle = b_i$ for $i=1,\dots,m$. For simplicity of presentation, we assume first that the rows of $A$ are linearly independent (and will address the linearly dependent case shortly).

We start by replacing each entry of $x^*$ by its closest rational number given some fixed denominator (say, $10^{40}$). 

In the approach of \cite{DLM21}, the system $Ax=b$ is first converted to row echelon form, after which the exact solution $\bar x$ is found using back substitution, where for the non-pivot columns, the corresponding entries of $x^*$ are used. If the conditioning of $A$ is not too bad, then $\bar x$ will be close to $x^*$.

The problem with this approach is that computing the row-echelon form of the system $Ax = b$ has to be done in exact arithmetic and becomes too expensive for large systems. In fact, depending on how the semidefinite program is formulated (for example, using low-rank constraint matrices as we do in this paper), forming the matrix $A$ can already be relatively expensive.

Another approach is to compute the solution $\bar x$ to $Ax=b$ closest to $x^*$ as follows. The minimum norm solution of $A\varepsilon = b - A x^*$ is given by $\varepsilon = A^+ (b - A x^*)$, where $A^+ = A^{\sf T} (AA^{\sf T})^{-1}$ is the pseudoinverse (which has this expression since $A$ is of full row rank). Then $\bar x = x^* + \varepsilon$, where $\varepsilon = A^{\sf T}y$ and $y$ is the solution to
\[
AA^{\sf T} y = b - A x^*,
\]
which can be computed efficiently using Dixon's algorithm \cite{D82}.

This approach, however, also requires the construction of the matrix $A$, and in practice we find that the determinant of $A A^{\sf T}$ can have large bit size, which leads to a solution $\bar x$ of large bit size. For example, with this approach the universal optimality proofs in Section~\ref{sec:universal} each have size around $11$GB, which is about a factor $500$ larger than with the following method.

We first build an invertible matrix $S$ by randomly selecting $r$ linearly independent columns of $A$, where $A \in \Q^{r \times s}$. To do so efficiently, we first select a submatrix containing $\ell$ random columns of $A$ with $r \le \ell \le s$, turn this submatrix into an integral matrix by clearing the denominators of each row, and then perform row reduction modulo some large prime number $p$. If the reduced matrix has $r$ pivots, these pivots yield $r$ linearly independent columns in $A$, and otherwise we retry with a different set of initial columns (for the problems considered in this paper we used $\ell=10r$ and never needed to retry). 

If the assumption that the rows of $A$ are linearly independent does not hold, it will be impossible to find $S$ with $r$ linearly independent columns. In that case, we must remove linearly dependent constraints before we can construct $S$. To find the constraints that need to be removed, we use column reduction modulo a prime.

This overall approach works best when $s$ is much larger than $r$, which is the case in our applications, because it is now no longer necessary to construct the matrix $A$: we can compute $\varepsilon$ by solving
\[
S\varepsilon = b - Ax^*
\]
using Dixon's method, and $Ax^*$ can often be computed efficiently without having to construct $A$ (in our applications we use the low-rank structure to do this).

Although this approach is very fast, the conditioning of $S$ tends to be bad, which leads to $\bar x$ being too far from $x^*$, in which case some of the strictly positive eigenvalues may become negative. We therefore interpolate between the two approaches described above, where the first approach leads to nearby solutions with large bit size, and the second to far away solutions of small bit size.

To interpolate we extend $S$ by some randomly chosen columns of $A$ (say, $r/10$ many columns). We can then compute $\varepsilon = S^+ (b - Ax^*)$ as $\varepsilon = S^{\sf T} y$, where $y$ is the solution to $S S^{\sf T} y = b-Ax^*$. This approach works well in practice.

\subsection{Rounding to algebraic numbers}\label{sec:notq}

Although the overall approach does not change when rounding over an algebraic number field $F$ of degree $d > 1$, some steps need small modifications. We give only a brief description of these modifications since the most important change is already considered in \cite{DLM21}, and for the proofs in this paper we need only $d=1$.

Note first that, since the reduced row-echelon form can be computed over any field, we can still find the kernel vectors by recognizing the entries of the reduced row-echelon form $M'$ of $M$ as elements of $F$ using the LLL algorithm. 

However, our reduction approach in Section~\ref{sec:optface} works only over $\Q$ due to the use of the Hermite normal form. We therefore apply the reduction method to the matrix $M'' \in \Q^{dk \times dn}$ obtained from $M' \in F^{k \times n}$ by using a basis of $F$ over $\Q$. The rows of $M''$ are linearly independent over $\Q$ if and only if the rows of $M'$ are linearly independent over $F$. After applying the reduction, we convert back to a matrix in $F^{dk \times n}$ and find a subset of $k$ linearly independent vectors over $F$ using floating point arithmetic.

To find a solution to the system $Ax=b$ close to $x^*$, with $A$, $b$, and $x$ defined over $F$, we use the same approach as \cite{DLM21}. Let $g$ be a generator of $F$, and consider the expansions $A = \sum_i A_i g^i$, $b = \sum_i b_i g^i$, and $x = \sum_i x_i g^i$. Then
\[
\sum_{i,j} A_i x_j g^{i+j} = \sum_l b_l g^l,
\]
which defines a system of equations $A'x' = b'$ with $A' \in \Q^{dr \times ds}$, and where $x'$ and $b'$ are the concatenations of $x_i$ and $b_i$ for $i=0, \dots, d-1$. A numerical approximation of $x'$ can be found by adding the constraints $x^* = \sum_i x_i g^i$ and solving the system using floating-point arithmetic, after which Section~\ref{sec:round} can be applied.

\subsection{Finding the algebraic number field}\label{sec:findfield}

Although for the problems in this paper it was always easy to guess over which field the optimal face can be defined, this is in general not the case. Therefore we propose the following heuristic to find the field.

Let $M'$ be the reduced row-echelon form of the matrix which has the numerical kernel vectors as rows. The entries of $M'$ are approximations of elements of the field $F$ that we want to find. For each entry we can use the LLL algorithm to find the minimal polynomial of the algebraic number it approximates. 

In the heuristic we first find the minimal polynomials for a selection of the entries of $M'$, and let $p$ be the minimal polynomial of highest degree and lowest bit size. Then we iterate through the same selection of entries and use the LLL algorithm to check whether each entry lies in the field with minimal polynomial $p$. If we find an entry that does not lie in the field, we replace $p$ by the minimal polynomial of a linear combination (typically just the sum) of an approximate root of $p$ and that entry. In practice, this method often gives the correct field $F$, even when considering only a small selection of entries of $M'$.

\subsection{General applicability of the rounding procedure}

Although our rounding procedure has been developed to obtain rigorous bounds in discrete geometry, we expect this approach to be useful more generally. To illustrate its use, we consider two examples from the literature where sums-of-squares characterizations are used, so that the semidefinite program does not have a unique optimal solution and the naive rounding approach mentioned in the introduction of this section cannot be used. Moreover, the optimal faces in these examples cannot be defined over $\Q$, so that we need Section~\ref{sec:findfield}. 

Consider the examples from \cite{PS01,GP04} and \cite{PY19} of finding the global minimum of the polynomial
\[
u^4 + v^4 + t^4 - 4uvt + u + v + t
\]
and the minimum of Caprasse's polynomial
\[
-ut^3+4vt^2w + 4 utw^2 + 2vw^3 + 4  ut+ 4t^2 -10 vw - 10  w^2 +2
\]
over the domain $[-1/2, 1/2]^4$.
Using semidefinite programming and the rounding procedure, the correct number fields (of degree $6$ and $2$, respectively) as well as exact minimizers are found in seconds. Interestingly, the minimal polynomials of the number field generators found by the procedure are much simpler than those of the minima in these optimization problems (smaller by a factor 1000 in terms of bit size).

\subsection{Implementation}

We have implemented the rounding procedure in Julia \cite{BEKS07} using the computer algebra system Nemo \cite{FHTJ17}. The code is available as part of the open-source semidefinite programming solver \texttt{ClusteredLowRankSolver.jl}, which is available on Github\footnote{\url{https://github.com/nanleij/ClusteredLowRankSolver.jl}} and can be installed as a Julia package; a snapshot of the git repository is also available at \cite{CdLL24}. The three-point bound for spherical codes is included as an example. The documentation can also be found there, including a tutorial.

\section{Computations}\label{sec:computation}

To prove optimality of spherical codes, we use the three-point bounds by Bachoc and Vallentin \cite{BV08}. To reduce the size of the semidefinite programs, we follow \cite{GP04,MO18} to exploit the symmetry in the polynomials. In the optimization problems we omit some of the variables that are not necessary to obtain the sharp bounds in this paper. To solve the semidefinite programs we use the solver from \cite{LL22}, which exploits the low-rank structure in the semidefinite programming formulations.

The three-point bound depends on the ambient dimension $n$, the minimal angle~$\theta$ of the spherical code, and a parameter $d$ we call the \emph{degree} of the bound (higher degrees yield improved bounds, at increased computational expense). For each $k$ from $0$ to $d$, let $Y_k^n(u,v,t)$ be the $(d-k+1) \times (d-k+1)$ matrix with entries
\[
Y_{k}^n(u,v,t)_{ij} = u^i v^j (1-u^2)^{k/2}(1-v^2)^{k/2} P_k^{n-1}\mathopen{}\left(\frac{t-uv}{\sqrt{(1-u^2)(1-v^2)}}\right)\mathclose{}
\]
for $0 \leq i,j \leq d-k$, where $P_k^n$ is the Gegenbauer polynomial of degree $k$ with parameter $n/2-1$, normalized so that $P_k^n(1) = 1$; note that it is convenient to index these entries starting with zero. Set 
\[
S_k^n = \frac{1}{6}\sum_{\sigma \in S_3} \sigma Y_k^n,
\]
where the permutation group $S_3$ permutes the three arguments of $Y_k^n$. 

The three-point bound involves optimizing over the choice of an auxiliary function $F$ defined by
\[
F(u,v,t) = \sum_{k=0}^d \langle F_k, S_k^n(u,v,t) \rangle,
\]
where $F_0,\dots,F_d$ are symmetric matrices with $F_k$ being a $(d-k+1) \times (d-k+1)$ matrix.  
The Bachoc-Vallentin three-point bound (with some variables removed) then reads
\begin{equation}
\label{eq:codebound}
\begin{aligned}
& \text{minimize} & & 1 + \langle F_0, J\rangle &&\\
& \text{subject to} & & 3F(u,u,1) \leq -1 && \text{for } u \in [-1, \cos \theta], \\
& & & F(u,v,t) \leq 0 && \text{for } (u,v,t) \in \Delta, \\
& & & F_0,\dots,F_d \succeq 0, & 
\end{aligned}
\end{equation}
where $J$ is the $(d+1) \times (d+1)$ all-ones matrix, 
\[
\Delta = \{(u,v,t): -1 \leq u,v,t \leq \cos \theta \text{ and } 1+2uvt - u^2-v^2-t^2 \geq 0\},
\]
and $M \succeq 0$ means $M$ is a positive-semidefinite matrix.

We can see as follows that the objective function $1 + \langle F_0, J\rangle$ is an upper bound for spherical code size. Let $\mathcal{C}$ be a spherical code in $S^{n-1}$ with minimal angle at least $\theta$, and fix a point $e \in S^{n-1}$. If $F_k \succeq 0$, then it follows from the addition formula for spherical harmonics that
\[
(x,y) \mapsto \langle F_k, Y_k^n(\langle x, e\rangle, \langle y, e\rangle, \langle x, y\rangle)\rangle
\]
is a positive semidefinite kernel. This implies that if $F$ is feasible (i.e., it satisfies the constraints in the optimization problem), then
\[
\sum_{x,y,z \in \mathcal{C}} F(\langle x, z \rangle, \langle y, z \rangle, \langle x, y \rangle) \geq 0.
\]
Separating terms and using the affine constraints gives
\begin{align*}
0 &\leq \sum_{x,y,z \in \mathcal{C}} F(\langle x, z \rangle, \langle y, z \rangle, \langle x, y \rangle) 
\\ &=|\mathcal{C}| F(1,1,1) + 3\sum_{\substack{x, y \in \mathcal{C}\\ x \neq y}}F(\langle x, y \rangle, \langle x,y\rangle, 1) + \sum_{\substack{x,y,z \in \mathcal{C}\\ \text{$x,y,z$ distinct}}} F(\langle x, z \rangle, \langle y, z \rangle, \langle x, y \rangle) \\
&\leq |\mathcal{C}|F(1,1,1) - |\mathcal{C}| (|\mathcal{C}|-1),
\end{align*}
from which it follows that
\begin{equation}
\label{eq:upperbound}
|\mathcal{C}| \leq 1+ F(1,1,1) = 1 + \langle F_0, J \rangle,
\end{equation}
because $S_k^n(1,1,1) = 0$ for $k>0$ and $S_0^n(1,1,1)=J$.

This argument shows that \eqref{eq:codebound} gives an upper bound for the maximal cardinality of a spherical code with minimal angle at least $\theta$. If this bound matches the code size, then the inequalities in the above derivation must hold with equality, which implies that $3F(\langle x, y \rangle, \langle x, y \rangle, 1) = -1$ for all distinct $x,y \in \mathcal{C}$. In other words, the zeros of the polynomial $3F(u,u,1)+1$ are the only inner products other than $1$ that can appear in a maximal code. This condition is a special case of complementary slackness in semidefinite programming.

So far, we have shown that every code $\mathcal{C}$ with minimal angle at least~$\theta$ must satisfy $|\mathcal{C}| \le 1 + \langle F_0, J\rangle$, but not yet that every code achieving this bound must be an optimal spherical code. To see why this must be the case, note first that the polynomial $3F(u,u,1)+1$ cannot vanish identically, since if it did, then \eqref{eq:upperbound} would imply that $|\mathcal{C}| \le 2/3$. Therefore complementary slackness shows that there are only finitely many possible inner products between distinct points in $\mathcal{C}$, namely the roots of $3F(u,u,1)+1$. If there were a code $\mathcal{C}$ with $|\mathcal{C}| = 1 + \langle F_0, J\rangle$ and minimal angle strictly greater than~$\theta$, then every sufficiently small perturbation of $\mathcal{C}$ would be a code with minimal angle at least~$\theta$, which would contradict the finite list of possible inner products. Thus, codes achieving equality in the three-point bound must be optimal codes.

To model the optimization problem~\eqref{eq:codebound} as a semidefinite program, we use sums-of-squares polynomials and symmetry reduction. Let $p(u) = (1+u)(\cos \theta - u)$, and let $m_\delta$ be a vector whose elements form a basis of the univariate polynomials up to degree $\delta$. Then the univariate constraint can be modeled as
\[
3F(u,u,1) + h_0(u)  + p(u) h_1(u) = -1,
\]
where $h_0(u) = \langle X_{1}, m_d m_d^{\sf T} \rangle$ and $h_1(u) = \langle X_2, m_{d-1} m_{d-1}^{\sf T} \rangle$ are sum-of-squares polynomials (in other words, $X_1$ and $X_2$ are positive semidefinite matrices).

To model the trivariate constraint efficiently using the symmetries, we describe the domain $\Delta$ as $\{(u,v,t) : s_i(u,v,t) \ge 0 \text{ for } i=1,\dots,4\}$, where
\begin{align*}
    s_1 &= p(u)+p(v)+p(t), & s_2 &= p(u)p(v) + p(v)p(t) + p(u)p(t), \\
    s_3 &= p(u)p(v)p(t), & s_4 &= 1+2uvt-u^2-v^2-t^2.
\end{align*}
Since these polynomials are $S_3$-invariant, the trivariate constraint can be enforced by setting
\[
F(u,v,t) + q_0 + \sum_{i=1}^4 s_i q_i = 0
\]
where $q_0,\dots,q_4$ are $S_3$-invariant sum-of-squares polynomials. In practice we use sum-of-squares polynomials such that $s_i q_i$ is of degree at most $2d$ for $i=0, \dots, 4$ (where we set $s_0 = 1$).

To describe the $S_3$-invariant sum-of-squares polynomials of degree $2d$, we use the formulation from \cite{GP04}, where the low-rank decomposition is more explicit than in \cite{MO18}. Let $b_{\delta}$ be a vector whose entries form a basis of the $S_3$-invariant polynomials of degree at most $\delta$. Let 
\begin{align*}
v_\mathrm{alt} &= (u-v)(v-t)(u-t),\\
v_{\mathrm{std},1} &= \frac{1}{\sqrt{2}}\begin{pmatrix}2u-v-t \\ 2vt - ut -uv\end{pmatrix}, \text{ and}\\
v_{\mathrm{std},2} &= \frac{\sqrt{3}}{\sqrt{2}}\begin{pmatrix}
    v-t \\ ut - uv
\end{pmatrix}.
\end{align*}
Define
\[
V_\mathrm{triv} = b_{d} b_{d}^{\sf T}\quad\text{and}\quad V_\mathrm{alt} = v_{\mathrm{alt}}^2 b_{d-3} b_{d-3}^{\sf T}
\]
and let $V_\mathrm{std}$ be the submatrix of
\[
(v_{\mathrm{std},1} \otimes b_{d-1}) (v_{\mathrm{std},1} \otimes b_{d-1})^{\sf T} + (v_{\mathrm{std},2} \otimes b_{d-1}) (v_{\mathrm{std},2} \otimes b_{d-1})^{\sf T} 
\]
in which the diagonal entries have degree at most $2d$ and $\otimes$ denotes the Kronecker product. Note that the square root factors in $v_{\mathrm{std}}$ get squared, so that the entries of this matrix are polynomials with rational coefficients.
Then every $S_3$-invariant sum of squares of polynomials $q$ of degree at most $2d$ is of the form
\[
q = \langle X_{\mathrm{triv}}, V_{\mathrm{triv}} \rangle + \langle X_{\mathrm{alt}}, V_{\mathrm{alt}} \rangle + \langle X_{\mathrm{std}}, V_{\mathrm{std}} \rangle
\]
for some positive definite matrices $X_{\mathrm{triv}}$, $X_{\mathrm{alt}}$, and $X_{\mathrm{std}}$. 

We solve the resulting semidefinite programs for the parameters listed in Table~\ref{table:computations}, where we also list the degrees and timing information. In each case, $|\mathcal{C}| = 1 + \langle F_0, J\rangle$, and therefore $\mathcal{C}$ must be optimal. Furthermore, in each case $3F(u,u,1)+1$ has no roots in $[-1,\cos \theta]$ other than the inner products that occur in $\mathcal{C}$; this observation will play a key role in proving uniqueness, because it shows that no other inner products can occur in any optimal code.

\begin{table}
\caption{The degree $d$ we use to obtain sharp three-point bounds, the number of digits of accuracy to which we solve the semidefinite programs, and the time in seconds it took to solve the semidefinite programs, round the solutions, and verify the proofs. For the last two cases, the three-point bound was already known to be sharp.}
\label{table:computations}
\begin{tabular}{rllcccc}
\toprule
$n$ & $\cos \theta$ & $d$ & Digits & Solve & Round & Check\\
\midrule
$16$ & $1/4$ & 8 & $40$ & $205$ & $4$ & $2$\\ 
$64$ & $1/8$ & 9 & $40$ & $484$ & $7$ & $6$\\
$256$ & $1/16$ & 10 & $40$ & $1075$ & $15$ & $17$\\ 
$1024$ & $1/32$ & 18 & $100$ & $219131$ & $6147$ & $3135$\\
\vspace{-0.6em}\\
$20$ & $1/15$ & 5 & $40$ & $9$ & $0.2$ & $0.1$\\
$21$ & $1/21$ & 5 & $40$ & $9$ & $0.2$ & $0.06$\\
$21$ & $1/12$ & 5 & $40$ & $9$ & $0.3$ & $0.06$\\
\vspace{-0.6em}\\
$55$ & $1/25$  & 5 & $40$ & $9$ & $0.2$ & $0.08$\\
$55$ & $1/22$ & 5 & $40$ & $10$ & $0.2$ & $0.07$\\
$56$ & $1/28$ &5 & $40$ & $10$ & $0.3$ & $0.07$\\
$56$ & $1/20$ & 5 & $40$ & $10$ & $0.3$ & $0.07$\\
$115$ & $3/115$ &5 & $40$ & $10$ & $0.2$ & $0.08$\\
$120$ & $1/45$ &5 & $40$ & $10$ & $0.3$ & $0.08$\\
$143$ & $1/65$ & 6 & $40$ & $34$ & $0.9$ & $0.6$\\
$1520$ & $1/456$ & 9 & $60$ & $586$ & $10$ & $17$\\
\vspace{-0.6em}\\
$3$ & $1/3$ & 10 & $40$ & $958$ & $23$ & $24$\\
\vspace{-0.6em}\\
$4$ & $1/6$ & 4 & $40$ & $4$ & $0.1$ & $0.03$\\
$3$ & $1/\big(\sqrt{8}+1\big)$ & 7 & $60$ & $112$ & $210$ & $100$\\
\bottomrule
\end{tabular}
\end{table}

The data files and verification script are available at \cite{CdLL24}. The verification code checks that the data matrices are positive semidefinite and that the affine constraints are satisfied exactly, and then prints the rigorous bound $1+\langle F_0,J\rangle$ and the zeros of $3F(u,u,1)+1$ in the interval $[-1,\cos \theta]$. To check that these are the only zeros in the interval, we use Sturm sequences. To verify that the matrices are positive semidefinite, we use a Cholesky factorization in ball arithmetic to check rigorously that the transformed variable matrices are positive definite. For checking the affine constraints, the matrices $S_k^n$ are constructed in the verification script, and we supply low-rank decompositions of the sum-of-squares constraint matrices so that nonnegativity of the sum-of-squares polynomials $h_0, h_1, q_0, \dots, q_4$ is immediate. To verify the correctness of the proofs, one needs to check the verification script and run the script on the data files. Table~\ref{table:computations} lists the verification times.

For the case of $24$ point in $\R^3$ (see \cite{R61}), the three-point bound seems to be sharp with degree $19$, but we have not been able to round the solution to an exact optimal solution. Here the required number field seems to be of degree $3$, and although our implementation does support general algebraic number fields, the precision required to find the right kernel vectors with our approach seems to be too high to be manageable so far.

\section{Universal optimality}
\label{sec:universal}

We prove not just universal optimality for the Nordstrom-Robinson spherical code, but also a strong form of uniqueness:

\begin{theorem} \label{theorem:univopt}
For each absolutely monotonic potential function $f \colon [-1,1) \to \R$, the Nordstrom-Robinson spherical code $\mathcal{C}$ minimizes the energy
\[
\sum_{\substack{x,y \in \mathcal{C}\\ x \ne y}} f(\langle x,y \rangle)
\]
over all subsets $\mathcal{C} \subseteq S^{15}$ with $|\mathcal{C}|=288$, and it is the unique minimizer up to isometry unless $f$ is a polynomial of degree at most~$5$.
\end{theorem}

By contrast, uniqueness fails when $f$ is a polynomial of degree~$4$ or less. To see why, recall that when $\mathcal{C}$ is a spherical $t$-design (i.e., when every multivariate polynomial of total degree at most $t$ has the same average over $\mathcal{C}$ as over the entire sphere), then it minimizes energy for every absolutely monotonic polynomial of degree at most $t$. The Nordstrom-Robinson code is a spherical $5$-design, and all we need to check is that there is a spherical $4$-design of the same size that is not isometric to it. To construct one, let $\mathcal{C}$ be the Nordstrom-Robinson spherical code as obtained from the binary code in Section~\ref{subsec:kerdock}, and for each point $(x_1,\dots,x_{16})$ in $\mathcal{C}$ with $x_{16}<0$, replace it with $(-x_1,-x_2,\dots,-x_{15},x_{16})$, while keeping the points with $x_{16} \ge 0$ the same. The result is not isometric to $\mathcal{C}$, since it has inner products $\pm 1$, $7/8$, $-3/8$, $\pm 1/4$, $\pm 1/8$, and $0$, compared with just $\pm 1$, $\pm 1/4$, and $0$ for $\mathcal{C}$, but one can check that it is still a $4$-design. For comparison, constructing a different $3$-design is simpler: if we decompose $\mathcal{C}$ into nine cross polytopes, then each cross polytope is itself a $3$-design, and they can be rotated independently. We do not know whether $\mathcal{C}$ is the unique $288$-point spherical $5$-design in $\R^{16}$ up to isometry, but the existence of another $5$-design is the only way uniqueness could fail for polynomials of degree~$5$ in Theorem~\ref{theorem:univopt}.

To prove Theorem~\ref{theorem:univopt}, we use the three-point bound for energy from \cite{CW12}, which works as follows. For studying $N$ points in $S^{n-1}$, instead of the matrices $S^n_k(u,v,t)$ from the previous section, we use
\[
T^n_k(u,v,t) = (N-2) S^n_k(u,v,t) + S^n_k(u,u,1) + S^n_k(v,v,1) + S^n_k(t,t,1).
\]
Suppose
\[
H(u,v,t) = c + \sum_{k = 0}^d \langle F_k, T^n_k(u,v,t) \rangle,
\]
where $c \in \R$ and $F_0,\dots,F_d$ are positive semidefinite matrices. If 
\[
H(u,v,t) \le (f(u)+f(v)+f(t))/3
\]
whenever
\[
-1 \leq u,v,t < 1 \quad\text{and}\quad 1+2uvt - u^2-v^2-t^2 \geq 0,
\]
then the energy of an $N$-point configuration in $S^{n-1}$ under the potential function $f$ is at least
\[
\frac{N}{2} \big( (N-1)c - \langle F_0, J \rangle\big).
\]
As in the case of spherical code bounds, expressing the inequality $H(u,v,t) \le (f(u)+f(v)+f(t))/3$ using sums of squares turns this optimization problem into a semidefinite program when $f$ is a polynomial. This semidefinite program can then be solved numerically, and the numerical solution can be rounded to obtain a rigorous proof.

One obstacle is that the cone of absolutely monotonic functions is infinite-dimensional and of course does not consist only of polynomials. To circumvent this issue, we use the techniques from Section~4 of \cite{CW12} to produce a finite basis for a slightly larger cone and prove optimality for all potential functions in that larger cone. Specifically, by Lemmas~9 and~10 in \cite{CW12}, each absolutely monotonic function $f$ on $[-1,1)$ is bounded below by a nonnegative linear combination of the polynomials 
{\allowdisplaybreaks\begin{align*}
&1,\quad 1+x,\quad (1+x)^2,\quad \dots,\quad (1+x)^{11},\\
&(1+x)^{11}(x+1/4),\\
&(1+x)^{11}(x+1/4)^2,\\
&(1+x)^{11}(x+1/4)^2x,\\
&(1+x)^{11}(x+1/4)^2x^2, \text{ and}\\
&(1+x)^{11}(x+1/4)^2x^2(x-1/4)
\end{align*}
}for which $\mathcal{C}$ has the same energy, namely the Hermite interpolation of $f$ to order~$11$ at $-1$ and order~$2$ at $\pm 1/4$ and $0$. (In other words, it is the unique polynomial $h$ of degree less than $11+2+2+2=17$ that $h^{(k)}(-1) = f^{(k)}(-1)$ for $k < 11$, $h^{(k)}(\pm 1/4) = f^{(k)}(\pm 1/4)$ for $k < 2$, and $h^{(k)}(0) = f^{(k)}(0)$ for $k < 2$.) In particular, if $\mathcal{C}$ minimizes the energy for each of these seventeen basis polynomials, then it is universally optimal. Because $\mathcal{C}$ is a spherical $5$-design, it automatically minimizes energy for $1$, $1+x$, $(1+x)^2$, \dots, $(1+x)^5$, and thus only eleven cases remain.

We can obtain a sharp three-point bound in each of these cases using degree $d=13$,  and rounding to rationals as described in Section~\ref{sec:roundingproc} yields a rigorous proof of optimality. The data set \cite{CdLL24} for this paper includes all the exact matrices as well as code for verifying that they work.

For each of the basis polynomials $p(x)$ except for those of degree up to~$5$, we in fact replace the polynomial $p(x)$ with $p(x)-(1+x)^{11}(x+1/4)^2x^2(x-1/4)^2/10000$ and prove a sharp bound for this modified polynomial. The reason is that for $t \in [-1,1)$,
\[
p(t)-(1+t)^{11}(t+1/4)^2t^2(t-1/4)^2/10000 \le p(t),
\]
with equality only for $t \in \{-1, \pm 1/4, 0\}$. Lowering the basis polynomials in this way ensures that equality can hold for the energy under $p$ only for codes with the same inner products as $\mathcal{C}$. The Nordstrom-Robinson code is the only such code up to isometry, by Theorem~\ref{thm:unique} below, and thus we obtain uniqueness for these basis polynomials, aside from those of degree at most~$5$.

To conclude that uniqueness holds in Theorem~\ref{theorem:univopt}, all we need to check is that if $f$ is an absolutely monotonic function on $[-1,1)$ whose Hermite interpolation $h$ to order~$11$ at $-1$ and order~$2$ at $\pm 1/4$ and $0$ has degree at most~$5$, then $f$ is itself a polynomial of degree at most~$5$, namely $h$. More generally, the only way $h$ can have degree less than~$16$ is if $f=h$. To see why, note that because $f-h$ has $17$ roots (counting with multiplicity), iterating Rolle's theorem shows that $(f-h)^{(16)}(t) = 0$ for some $t \in (-1,1)$. Because $h^{(16)}=0$ by assumption, we must have $f^{(16)}(t)=0$, which implies that $f^{(16)}$ vanishes identically on $[-1,t)$ by absolute monotonicity. By Theorem~3a in Chapter~IV of \cite{W41}, $f$ is analytic and therefore $f^{(16)}=0$ everywhere. Thus, $f$ must be a polynomial of degree less than~$16$ and is therefore equal to~$h$, because $f-h$ has $17$ roots.

\section{Uniqueness}
\label{sec:uniqueness}

We have shown in Section~\ref{sec:computation} that each of the codes listed in Table~\ref{table:newcases} is optimal, and that the only possible angles that can occur between distinct points are those listed in the table. Furthermore, we have proved universal optimality in Section~\ref{sec:universal}. All that remains for the proof of Theorem~\ref{theorem:main} is to prove uniqueness for dimensions up to~$64$.

\begin{theorem} \label{thm:unique}
There is a unique optimal spherical code in $\R^{16}$ with $288$ points, up to isometry. 
\end{theorem}

The proof simply amounts to combining the complementary slackness conditions from our bound with the results of \cite{CCKS97} and \cite{S73}.

\begin{proof}
Let $\mathcal{C}$ be such a code. As shown in Section~\ref{sec:computation}, complementary slackness implies that the only possible inner products that can occur between the points of $\mathcal{C}$ are those that occur in the Nordstrom-Robinson spherical code, namely $0$, $\pm 1/4$, and $\pm 1$. Furthermore, $\mathcal{C}$ must be antipodal (closed under multiplication by $-1$), because if $x$ were a point of $\mathcal{C}$ such that $-x \not\in \mathcal{C}$, then $\mathcal{C} \cup \{-x\}$ would have maximal inner product $1/4$ and size $289$, which contradicts the three-point bound. Thus, the points of $\mathcal{C}$ lie on $144$ lines through the origin, with angles $\pi/2$ and $\arccos(1/4)$ between them.

Now we can apply Proposition~3.12 from \cite{CCKS97}, which says that these $144$ lines can be partitioned into $9$ sets of $16$ orthogonal lines, with pairs of lines from different sets not being orthogonal. If we pick one of these sets of $16$ orthogonal lines and use it to define our coordinate system, then $32$ points of $\mathcal{C}$ lie on these coordinate axes, and the remaining $256$ points have all their coordinates equal to $\pm 1/4$. The signs of these coordinates define a binary code $\widetilde{\mathcal{C}} \subseteq \{0,1\}^{16}$; specifically, $c \in \widetilde{\mathcal{C}}$ iff the point $((-1)^{c_1},\dots,(-1)^{c_{16}})/4$ is in $\mathcal{C}$. For $c, c' \in \widetilde{\mathcal{C}}$, the inner product of the corresponding points in $\mathcal{C}$ is $1-d(c,c')/8$, where $d(c,c') = |\{i : c_i \ne c'_i\}|$ is the Hamming distance between $c$ and $c'$. Thus, we have a binary code $\widetilde{\mathcal{C}}$ with $|\widetilde{\mathcal{C}}| = 256$ and minimal Hamming distance $6$. However, the Nordstrom-Robinson code is the only such code by Theorem~10.2.1 in \cite{S73}, in the sense that any other code with these parameters is obtained by translating and permuting the coordinates. These operations amount to isometries of $\R^{16}$, and thus $\mathcal{C}$ is unique up to isometry.
\end{proof}

The same argument also proves uniqueness in $64$ dimensions, using the proof from \cite{P15} of uniqueness of the Kerdock code of block length~$64$ given its weight enumerator:

\begin{theorem}
There is a unique optimal spherical code in $\R^{64}$ with $4224$ points, up to isometry. 
\end{theorem}

\begin{proof}
Let $\mathcal{C}$ be such a code. As in the previous proof, $\mathcal{C}$ must be antipodal, and Proposition~3.12 from \cite{CCKS97} says that the $2112$ lines through points in $\mathcal{C}$ can be partitioned into $33$ sets of $64$ orthogonal lines, with pairs of lines from different sets not being orthogonal. We again choose coordinates corresponding to one of the $33$ sets of orthogonal lines, and use the remaining $32$ sets to define a binary code $\widetilde{\mathcal{C}} \subseteq \{0,1\}^{64}$ with $|\widetilde{\mathcal{C}}|=4096$. Based on the decomposition into sets of orthogonal lines, we know that each codeword in $\widetilde{\mathcal{C}}$ has $1984$ neighbors at Hamming distance~$28$, $126$ at distance~$32$, $1984$ at distance $36$, and $1$ at distance $64$. Finally, Theorem~1 in \cite{P15} shows that such a code is unique up to translating and permuting the coordinates, as desired.
\end{proof}

In exactly the same way, we can deduce uniqueness of the Kerdock binary code of block length $64$:

\begin{theorem}
There is a unique binary code of block length~$64$ with $4096$ codewords and minimal distance~$28$, up to translating and permuting the coordinates.
\end{theorem}

It seems not to be known whether the Kerdock code of block length $256$ is unique. However, uniqueness is known to fail for block length $1024$, and more generally for block length $2^{2n}$ whenever $2n-1$ is composite, using a construction from \cite{K82a,K82b}. Furthermore, Proposition~3.16 from \cite{CCKS97} shows that the corresponding spherical codes are not isometric to each other.

All that remains in the proof of Theorem~\ref{theorem:main} is to prove uniqueness for the codes constructed using triangle-free strongly regular graphs.

\begin{theorem}
There are unique optimal spherical codes with $56$ points in $\R^{20}$, $50$~points in $\R^{21}$, and $77$ points in $\R^{21}$, up to isometry.
\end{theorem}

\begin{proof}
As shown in Section~\ref{sec:computation} using complementary slackness, an optimal code in one of these cases must be a two-distance set, with the distances specified in Table~\ref{table:newcases}. It therefore defines a graph, with vertex set equal to the code and edges corresponding to the larger distance (i.e., smaller inner product). If we can show that this graph is what we expect, namely the Gewirtz, Hoffman-Singleton, or $M_{22}$ graph, then the code is unique up to isometry, because any permutation of the points in the code that preserves all the pairwise distances corresponds to an isometry of the Euclidean space they span.

Each of these graphs is the only strongly regular graph with its parameters \cite{G69b,BH93,HS60,B83}, up to isomorphism. It will suffice to show that the code must be a spherical $2$-design, because by Theorem~7.4 in \cite{DGS77}, any spherical $2$-design that is a two-distance set must yield a strongly regular graph whose parameters are uniquely determined by the dimension, size, and angles of the design.

Let $\mathcal{C} \subseteq \R^n$ be the code in question, with $|\mathcal{C}|=N$ and inner products $\alpha$ and $\beta$ between distinct points such that $\alpha<\beta$, and let $k = |\{(x,y) \in \mathcal{C}^2 : \langle x,y \rangle = \alpha\}|/N$. Then the characterization of spherical designs in terms of Gegenbauer polynomial sums (Theorem~5.5 in \cite{DGS77}) shows that
\[
k \alpha + (N-k-1)\beta + 1 \ge 0 \quad\text{and}\quad k(\alpha^2-1/d) + (N-k-1)(\beta^2-1/d) + (1^2-1/d) \ge 0,
\]
with equality if and only if $\mathcal{C}$ is a spherical $2$-design. In each of our three cases, $N$, $n$, $\alpha$, and $\beta$ are determined, as in Table~\ref{table:newcases}, and one can check that the only way for $k$ to satisfy these two inequalities is if both of them are equalities. (The reason is that the coefficients of $k$ in these inequalities are $\alpha-\beta$ and $\alpha^2-\beta^2$, which have opposite signs in our cases because $\alpha < -\beta < 0 < \beta$. Thus, any value of $k$ that makes both inequalities sharp must be the unique solution.) We conclude that $\mathcal{C}$ must be a spherical $2$-design and is therefore unique, as desired.
\end{proof}

The same argument as in the previous proof shows that any spherical code with the hypothetical parameters shown in Table~\ref{table:hypotheticalcases} must be a spectral embedding of a strongly regular graph. However, it is unknown whether such graphs would be unique if they exist.

\section{Conjecturally optimal codes}
\label{sec:conjectures}

Our results in this paper show that three-point bounds are sharp more often than was previously known, and it seems likely that the same is true for four-point bounds and beyond. It is natural to ask which spherical codes might be provably optimal using $k$-point bounds for small $k$, but no clear pattern seems to predict sharp cases. In the absence of such a pattern, we instead give a list of possibilities in Table~\ref{table:smallcodes} as targets for future bounds. Each of these codes has at most a hundred points and at least $2n+1$ points in $n$ dimensions, has a transitive symmetry group, and appears to optimize both the minimal distance given the number of points and the number of points given the minimal distance. We expect that Table~\ref{table:smallcodes} is not a complete list of such codes, but we are not aware of any other candidates for which optimality has not yet been proved. See \cite{codetables} for numerical coordinates for those codes, and Appendix~\ref{app:constructions} for exact coordinates in the cases not covered by the references cited in the table.

\begin{table}
\caption{Conjecturally optimal spherical codes with at most a hundred points and transitive symmetry groups.
Each code consists of $N$ points in $S^{n-1}$ with minimal angle $\theta$, $m$ distances
between distinct points, strength $t$ as a spherical design,
and automorphism group $G$. For coordinates, see \cite{codetables}.}
\label{table:smallcodes}
\begin{tabular}{rllrrrl}
\toprule
$n$ & $N$ & $\cos \theta$ & $m$ & $t$ & $|G|$ & References\\
\midrule
$4$ & $24$ & $1/2$ & $4$ & $5$ & $1152$ & \cite{S01}\\
$4$ & $26$ & $0.54078961\dots$ & $9$ & $3$ & $52$ & \cite{SHS,N95}\\
$4$ & $50$ & $0.69372612\dots$ & $18$ & $3$ & $100$ & \cite{SHS,N95}\\
$5$ & $40$ & $1/2$ & $4$ & $3$ & $3840$ & \cite{KZ73}\\
$6$ & $20$ & $3/14$ & $7$ & $1$ & $20$ & \cite{W09}\\
$6$ & $32$ & $1/3$ & $3$ & $3$ & $23040$ & \cite{LS66}\\
$6$ & $72$ & $1/2$ & $4$ & $5$ & $103680$ & \cite{KZ73}\\
$7$ & $18$ & $1/\big(4\sqrt{3}+1\big)$ & $3$ & $1$ & $216$ & \cite{A97}\\
$7$ & $24$ & $1/\big(2\sqrt{5}+1\big)$ & $4$ & $1$ & $480$ & \cite{W09}\\
$8$ & $64$ & $1/3$ & $4$ & $3$ & $12288$ & \cite{CHS96,EZ01}\\
$8$ & $72$ & $5/14$ & $5$ & $3$ & $725760$ & \cite{CHS96}\\
$9$ & $32$ & $1/\big(4\sqrt{2}+1\big)$ & $3$ & $1$ & $43008$ & \cite{EZ01}\\
$9$ & $36$ & $1/5$ & $3$ & $3$ & $1440$ & \cite{EZ01}\\
$9$ & $96$ & $1/3$ & $4$ & $3$ & $221184$ & \cite{CHS96}\\
$10$ & $40$ & $1/6$ & $4$ & $3$ & $1920$ & \cite{SHS}\\
$10$ & $66$ & $4/15$ & $3$ & $2$ & $7920$ & \cite{EZ01}\\
$11$ & $30$ & $1/9$ & $4$ & $1$ & $720$ & \cite{C22}\\
$12$ & $26$ & $1/14$ & $5$ & $1$ & $52$ & \cite{EZ01}\\
$12$ & $78$ & $1/4$ & $4$ & $3$ & $11232$ & \cite{CHS96}\\
$13$ & $32$ & $1/\big(4\sqrt{5}+5\big)$ & $3$ & $1$ & $1920$ & \cite{SHS}\\
$13$ & $36$ & $1/\big(4\sqrt{6}+1\big)$ & $3$ & $1$ & $4320$ & \\
$13$ & $42$ & $1/8$ & $3$ & $1$ & $336$ & \cite{EZ01}\\
$13$ & $48$ & $1/\big(4\sqrt{3}+1\big)$ & $3$ & $1$ & $380160$ & \cite{EZ01}\\
$14$ & $42$ & $1/10$ & $3$ & $2$ & $5040$ & \cite{EZ01}\\
$14$ & $64$ & $1/7$ & $3$ & $3$ & $21504$ & \cite{EZ01}\\
$14$ & $70$ & $1/6$ & $3$ & $2$ & $5040$ & \cite{EZ01}\\
$16$ & $34$ & $1/18$ & $3$ & $2$ & $272$ & \cite{EZ01}\\
$17$ & $64$ & $1/9$ & $3$ & $1$ & $20643840$ & \cite{EZ01}\\
$18$ & $96$ & $1/7$ & $3$ & $1$ & $2211840$ & \cite{EZ01}\\
$19$ & $40$ & $1/21$ & $2$ & $1$ & $7680$ & \cite{EZ01}\\
$19$ & $60$ & $1/11$ & $2$ & $1$ & $5760$ & \\
$20$ & $42$ & $1/22$ & $2$ & $1$ & $5040$ & \cite{EZ01}\\
$20$ & $60$ & $1/13$ & $2$ & $1$ & $23040$ & \\
$20$ & $81$ & $1/10$ & $2$ & $2$ & $233280$ & \cite{BH92}\\
$24$ & $50$ & $1/26$ & $3$ & $2$ & $1200$ & \cite{EZ01}\\
\bottomrule
\end{tabular}
\end{table}

We see no reason to believe that candidates for sharp $k$-point bounds should necessarily have transitive symmetry groups. For example, the projective code analyzed in \cite{CW12} using three-point bounds does not. However, codes with transitive symmetry groups are among the best-behaved codes, and this condition keeps Table~\ref{table:smallcodes} from becoming excessively long.

We conjecture that all the codes listed in Table~\ref{table:smallcodes} are optimal, but of course it is difficult to say for certain. We expect that most of them are indeed optimal, but all we can say with confidence is that it seems difficult to improve on each of them. Two of these codes, with $40$ points in $\R^{10}$ and with $64$ points in $\R^{14}$, are conjectured to be universally optimal \cite{BBCGKS09}, but the others are not universally optimal.

Some of the codes listed in Table~\ref{table:smallcodes} may be considerably more difficult to handle than others. For example, the $D_4$ root system (with $24$ points in $\R^4$) appears to be unique, but the $D_5$ root system (with $40$ points in $\R^5$) has at least two equally good spherical codes that are not isometric to it \cite{L67b,S23}. The complementary slackness conditions for $D_5$ would therefore have to take into account all three codes. The most dramatic case may be $72$ points in $\R^8$, which is not rigid. This code is most naturally constructed in the subspace of $\R^9$ with coordinate sum zero. The code is the union $\mathcal{C} \cup -\mathcal{C}$, where $\mathcal{C}$ is the $36$-point code consisting of the permutations of $(7, 7, -2, -2, -2, -2, -2, -2, -2)/\sqrt{126}$. (Equivalently, $\mathcal{C}$ consists of the midpoints of the edges in a regular simplex.) The maximal inner product of $\mathcal{C} \cup -\mathcal{C}$ is $5/14$, but the maximal inner product between $\mathcal{C}$ and $-\mathcal{C}$ is just $2/7$. In other words, $\mathcal{C}$ and $-\mathcal{C}$ are not adjacent, and either one can be perturbed arbitrarily without getting too close to the other, as long as the perturbation is small enough. This non-rigidity is analogous to that of the fluid diamond packing in \cite{CS95}, and it would greatly constrain a proof of optimality using $k$-point bounds. By contrast, we believe that every other code listed in Table~\ref{table:smallcodes} is rigid. Their rigidity could likely be verified rigorously using the techniques of \cite{CJKT11}, but we have not done so.

Another potential complication is the need for algebraic numbers. In particular, the cosine of the minimal angle in the code with $26$ points in $\R^4$ has minimal polynomial
\[
3392x^6+2112x^5-496x^4-656x^3-132x^2+6x-1,
\]
while for $50$ points in $\R^4$ the minimal polynomial is
\[
\begin{split}
308224x^{10}&+168960x^9-94720x^8-119680x^7-21760x^6\\
\phantom{}&+10752x^5+3040x^4+120x^3+80x^2+10x-1.
\end{split}
\]
For comparison, all the other cases in Table~\ref{table:smallcodes} have minimal polynomials of degree at most~$2$. Having to use algebraic numbers certainly does not preclude a proof, but it can make the computations more cumbersome.

One of the most remarkable codes listed in Table~\ref{table:smallcodes} is the code with $81$ points in $\R^{20}$ and maximal inner product $1/10$. This code can be obtained as a spectral embedding of the unique strongly regular graph with parameters $(81,20,1,6)$ (see \cite{BH92}), which is not triangle-free. However, this code behaves in many ways like the Hoffman-Singleton, Gewirtz, and $M_{22}$ codes.
For example, it is the kissing configuration of the optimal code with $112$ points in $\R^{21}$, and the optimal code with $162$ points in $\R^{21}$ is an antiprism consisting of two copies of this code on parallel planes (see \cite[Section~3.7]{BBCGKS09}).

\appendix
\section{Two-point bounds}
\label{app:two-point}

\begin{table}
\caption{Two-point bounds for the codes from Table~\ref{table:newcases}.}
\label{table:LPbounds}
\begin{tabular}{rrlll}
\toprule
$n$ & $N$ & $\cos \theta$ & Two-point bound for $N$\\
\midrule
$16$ & $288$ & $1/4$ & $312$\\
$20$ & $56$ & $1/15$ & $69.6585\dots$\\
$21$ & $50$ & $1/21$ & $64$\\
$21$ & $77$ & $1/12$ & $86.3902\dots$\\
$64$ & $4224$ & $1/8$ & $4448$\\
$256$ & $66048$ & $1/16$ & $67968$\\
$1024$ & $1050624$ & $1/32$ & $1066496$\\
\bottomrule
\end{tabular}
\end{table}

\begin{table}
\caption{Dual two-point bounds matching Table~\ref{table:LPbounds}, using the notation from Appendix~\ref{app:two-point}.}
\label{table:LPbounds2}
\begin{tabular}{rllll}
\toprule
$n$ & Cosines $t_1,\dots,t_m$ & Weights $w_1,\dots,w_m$ \\
\midrule
$16$ & $-\frac{1}{2}\pm\frac{\sqrt{3}}{6}$, $\frac{1}{4}$, $1$ & $\frac{3825}{46}\pm\frac{2115}{46}\sqrt{3}$, $\frac{3328}{23}$, $1$\\[0.15cm]
$20$ & $-\frac{16}{35}$, $\frac{1}{15}$, $1$ & $\frac{4802}{451}$, $\frac{26163}{451}$, $1$\\[0.15cm]
$21$ & $-\frac{11}{21}$, $\frac{1}{21}$, $1$ & $7$, $56$, $1$\\[0.15cm]
$21$ & $-\frac{13}{33}$, $\frac{1}{12}$, $1$ & $\frac{14641}{861}$, $\frac{58880}{861}$, $1$\\[0.15cm]
$64$ & $-\frac{1}{2}\pm\frac{\sqrt{627}}{66}$, $\frac{1}{8}$, $1$ & $\frac{1178199}{1042}\pm\frac{892485}{19798}\sqrt{627}$, $\frac{1138688}{521}$, $1$ \\[0.15cm]
$256$ & $-\frac{1}{2}\pm\frac{\sqrt{1419}}{86}$, $\frac{1}{16}$, $1$ & $\frac{15594423}{914}\pm\frac{4553399}{10054}\sqrt{1419}$, $\frac{15466496}{457}$, $1$ \\[0.15cm]
$1024$ & $-\frac{1}{2}\pm\frac{\sqrt{231363}}{1026}$, $\frac{1}{32}$, $1$ & $\frac{17508632031}{65602}\pm\frac{1492406775}{2689682}\sqrt{231363}$, $\frac{17473470464}{32801}$, $1$\\
\bottomrule
\end{tabular}
\end{table}

Table~\ref{table:LPbounds} shows that the two-point bound is not sharp for the spherical codes listed in Table~\ref{table:newcases}, and Table~\ref{table:LPbounds2} provides the information needed to verify the two-point bounds in Table~\ref{table:LPbounds} and their optimality. Recall that the \emph{two-point bound} for spherical codes in $S^{n-1}$ with minimal angle $\theta$ is given by the infimum of $f(1)/f_0$ over all single-variable polynomials $f$ with real coefficients such that $f(x) \le 0$ for $-1 \le x \le \cos \theta$ and $f(x) = \sum_{k=0}^d f_k C_k^{n/2-1}(x)$ with nonnegative coefficients $f_i$ (and $f_0>0$), where the polynomials $C_k^{\lambda}$ are the Gegenbauer polynomials with generating function
\[
\sum_{k=0}^\infty C_k^\lambda(x)z^k = \frac{1}{(1-2xz + z^2)^\lambda}.
\]
Conversely, given $-1 \le t_1 < t_2 < \dots < t_m = 1$ and weights $w_1,\dots,w_m \ge 0$ with $t_{m-1} = \cos \theta$ and $w_m = 1$, we obtain a \emph{dual two-point bound} of $\sum_{j=1}^m w_j$ if the inequality
\[
\sum_{j=1}^m w_j C^{n/2-1}_k(t_j) \ge 0
\]
holds for all $k \ge 0$. Note that these inequalities always imply that $f(1)/f_0 \ge \sum_{j=1}^m w_j$. It follows that when $f(1)/f_0 = \sum_{j=1}^m w_j$, we can conclude that both bounds are fully optimized.

Table~\ref{table:LPbounds2} lists the values of $t_1,\dots,t_m$ and $w_1,\dots,w_m$ for each case in Table~\ref{table:LPbounds}. The corresponding polynomial $f(x)$ is then given by
\[
f(x) = (x-t_1)^2 \dots (x-t_{m-2})^2 (x-t_{m-1}).
\]
An explicit calculation shows that the coefficients $f_0,\dots,f_d$ are nonnegative and that $f(1)/f_0 = \sum_{j=1}^m w_j$. For the dual bound, we need to check the inequalities
\[
\sum_{j=1}^m w_j C^{n/2-1}_k(t_j) \ge 0
\]
for each $k \ge 0$. Reducing this verification to checking finitely many cases is the crux of the proof.

To do so, we will obtain bounds for the Gegenbauer polynomials by using the Cauchy residue theorem to write
\[
C^{n/2-1}_k(t) = \frac{1}{2\pi i} \int_{|z| = 1-1/k} \frac{dz}{(1-2tz+z^2)^{n/2-1} z^{k+1}}.
\]
If $-1 < t < 1$, then the integrand has poles at $r_{\pm} = t \pm i\sqrt{1-t^2}$ and $0$. We can factor $1-2tz+z^2$ as $(z-r_+)(z-r_-)$. Because $|r_+-r_-| = 2\sqrt{1-t^2}$, the greater of $|z-r_-|$ and $|z-r_+|$ is always at least $\sqrt{1-t^2}$, while the lesser is at least $1/k$ for $|z|=1-1/k$ since $|r_\pm|=1$. Thus,
\[
|C_k^{n/2-1}(t)| \le \frac{k^{n/2-1}}{(1-1/k)^k (1-t^2)^{(n-2)/4}}.
\]
In particular, for $k \ge 6$ and $-1 < t < 1$, 
\[
|C_k^{n/2-1}(t)| \le \frac{3k^{n/2-1}}{(1-t^2)^{(n-2)/4}}.
\]
This bound is somewhat crude, but it has the advantage of being simple and explicit. The true growth rate is $k^{n/2-2}$, and this improvement would be important for $n \le 4$, but in our cases it is not needed.

For comparison, the generating function implies that
\[
C_k^{n/2-1}(1) = \binom{n+k-3}{n-3},
\]
which grows proportionally to $k^{n-3}$ as $k \to \infty$. Because $n>4$ and $t_1 \ne -1$, the upper bound for $|C_k^{n/2-1}(t)|$ grows more slowly than $k^{n-3}$ as $k$ grows. Thus,
\[
\sum_{j=1}^m w_j C^{n/2-1}_k(t_j) \sim C^{n/2-1}_k(1) = \binom{n+k-3}{n-3}
\]
as $k \to \infty$, and so the sum must be nonnegative in all but finitely many cases. Specifically, it is nonnegative for each $k \ge 6$ that satisfies
\[
\binom{n+k-3}{n-3} \ge \sum_{j=1}^{m-1} \frac{3w_j}{(1-t_j^2)^{(n-2)/4}}k^{n/2-1}. 
\]
If we use the lower bound
\[
\binom{n+k-3}{n-3} \ge \frac{k^{n-3}}{(n-3)!},
\]
we find that the desired inequality holds whenever
\[
k^{n/2-2} \ge (n-3)! \sum_{j=1}^{m-1}\frac{3w_j}{(1-t_j^2)^{(n-2)/4}}
\]
and $k \ge 6$. Checking it for all smaller $k$ via explicit computation with interval arithmetic completes the proof of validity for the dual bound. In the cases listed in Table~\ref{table:LPbounds2}, the worst case is $n=1024$, for which we must check up through $k = 583999$ directly. The data set \cite{CdLL24} includes computer code for this verification.

\section{Constructions of spherical codes}
\label{app:constructions}

There are three codes listed in Table~\ref{table:smallcodes} for which we do not know of prior references, and three for which the references we know do not give exact constructions. In this appendix we construct these codes.

\subsection{$26$ points in $\R^4$}
Let $\alpha = 0.210938\dots$ be a root of
\[
\begin{split}
2809x^{12} &- 16854x^{11} + 42689x^{10} - 58950x^9 + 47966x^8 - 23558x^7\\
\phantom{} &+ 7305x^6 - 1658x^5 + 250x^4 - 18x^3 + 30x^2 - 11x + 1.
\end{split}
\]
Then the $26$ points are the orbits of the column vector $\big(\sqrt{\alpha},0,-\sqrt{1-\alpha},0\big)$ under the group generated by the matrices
\[
\begin{bmatrix}
\cos 2\pi/13  & -\sin2\pi/13 & 0 & 0\\
\sin 2\pi/13 & \cos 2\pi/13 & 0 & 0\\
0 & 0 & \cos 10\pi/13 & -\sin 10\pi/13\\
0 & 0 & \sin 10\pi/13 & \cos 10\pi/13
\end{bmatrix}
\quad\text{and}\quad
\begin{bmatrix}
0 & 0 & 1 & 0\\
0 & 0 & 0 & 1\\
1 & 0 & 0 & 0\\
0 & -1 & 0 & 0
\end{bmatrix}.
\]

\subsection{$50$ points in $\R^4$}
Let $\alpha = 0.237772\dots$ be a root of
\[
\begin{split}
90601x^{20} &- 906010x^{19} + 2668465x^{18} + 1805100x^{17} - 24135620x^{16} + 35112868x^{15}\\
\phantom{} &+ 52674165x^{14} - 254707285x^{13} + 459094925x^{12} - 660243745x^{11}\\
\phantom{} &+ 962656504x^{10} - 1212562570x^9 + 1105685255x^8 - 686386675x^7\\
\phantom{} &+ 283874785x^6 - 76272968x^5 + 12691245x^4 - 1189505x^3\\
\phantom{} &+ 51000x^2 - 535x + 1.
\end{split}
\]
Then the $50$ points are the orbits of the column vector $\big(\sqrt{\alpha},0,-\sqrt{1-\alpha},0\big)$ under the group generated by the matrices
\[
\begin{bmatrix}
\cos 2\pi/25  & -\sin2\pi/25 & 0 & 0\\
\sin 2\pi/25 & \cos 2\pi/25 & 0 & 0\\
0 & 0 & \cos 14\pi/25 & -\sin 14\pi/25\\
0 & 0 & \sin 14\pi/25 & \cos 14\pi/25
\end{bmatrix}
\quad\text{and}\quad
\begin{bmatrix}
0 & 0 & 1 & 0\\
0 & 0 & 0 & 1\\
1 & 0 & 0 & 0\\
0 & -1 & 0 & 0
\end{bmatrix}.
\]

\subsection{$18$ points in $\R^7$}
We identify $\R^{6}$ with $\C^3$; for each $x \in \C^3$ and $y \in \R$, we write $(x,y)$ for their concatenation in $\R^{7}$. Let $(\alpha_1,\dots,\alpha_4)$ be
\[
\left(-\sqrt{\frac{-277+168\sqrt{3}}{47}},-\sqrt{\frac{-331+196\sqrt{3}}{47}},
-\sqrt{\frac{-55+32\sqrt{3}}{47}},\sqrt{\frac{-1+4\sqrt{3}}{47}}\right),
\]
and let $\beta_i = \sqrt{1-\alpha_i^2}$; we also set $\omega = -1/2 + \sqrt{3}i/2$. Given a vector $x \in \C^3$, let $f_i(x)$ be the concatenation $(\beta_i x/|x|,\alpha_i)$, which is a unit vector in $\R^7$. Then the $18$-point code in $\R^7$ consists of $f_1(x)$ for all permutations $x$ of $(1-i,1-i,(1+i)\omega)$, $f_2(x)$ for all permutations $x$ of $\big(\omega+1,\omega+1,-(\omega+1)/\sqrt{3}\big)$, $f_3(x)$ for all permutations $x$ of $\big(\omega-i,\omega-i,\big(2-\sqrt{3}\big)(\omega+1)\big)$, and $f_4(x)$ for all permutations $x$ of one of $(1,0,0)$, $(-i,0,0)$, or $(\omega,0,0)$.

\subsection{$36$ points in $\R^{13}$}
We found this code via numerical optimization, after which examining its structure led to the following construction. Let $\alpha = \sqrt{(-1+4\sqrt{6})/95}$ and $\beta = \sqrt{1-\alpha^2}$. Each point in the code will have $\pm\alpha$ as its thirteenth coordinate. To describe the other twelve coordinates, we identify $\R^{12}$ with $\C^6$; for each $x \in \C^6$ and $y \in \R$, we write $(x,y)$ for their concatenation in $\R^{13}$. Let $e_1,\dots,e_6$ be the standard basis of $\C^6$, and let $\omega$ be a primitive cube root of unity. Then $18$ of the points of the code are $(\omega^i \beta e_j, \alpha)$ for $0\le i \le 2$ and $1 \le j \le 6$. To construct the remaining $18$ points, we will use the hexacode, which is a linear code over the finite field $\F_4$. We abuse notation by writing $\F_4 = \{0, 1, \omega, \omega^2\}$; these are of course not the complex numbers with the same name, but we will identify them with those complex numbers shortly. In this notation, the hexacode is the subspace of $\F_4^6$ consisting of the points
\[
(a,b,c, a+b+c, a+b\omega+c\omega^2, a+b\omega^2+c\omega)
\]
with $a,b,c \in \F_4$. One can check that there are $18$ points in the hexacode for which all six coordinates are nonzero. We identify these $18$ points with the notationally identical vectors $v \in \C^6$. Then the remaining $18$ points in the spherical code are the points $\big(-\beta v/\sqrt{6},-\alpha\big)$.

\subsection{$60$ points in $\R^{19}$ or $\R^{20}$}
These codes are the kissing configurations of the code with $81$ points in $\R^{20}$ and the $M_{22}$ code, respectively. In other words, they consist of the nearest neighbors of any points in those codes.

\end{document}